\RequirePackage{fix-cm}
\documentclass{svjour3}                     
\smartqed  
\usepackage{graphicx}
\makeatletter
\def\ps@pprintTitle{%
	\let\@oddhead\@empty
	\let\@evenhead\@empty
	\def\@oddfoot{}%
	\let\@evenfoot\@oddfoot}
\makeatother

\usepackage{lineno}
\modulolinenumbers[5]
\usepackage{amsfonts,graphicx,amssymb, booktabs}
\usepackage{multirow}
\usepackage[linesnumbered,ruled,vlined,noresetcount]{algorithm2e}
\usepackage{amsmath}
\usepackage{mathrsfs}
\usepackage{bbm}
\usepackage{hyperref} 
\usepackage{xcolor}
\usepackage{enumitem}
\usepackage{array}

\usepackage[round, sort, authoryear]{natbib}
\newcommand{\ii}{{\vec{i}}}
\newcommand{\tightindex}{( { j_1^{\ii}, j_2^{\ii}, \dots, j^{\ii}_{l^{\ii}} }) }
\newcommand{\tightpower}{( {\alpha_1^{\ii}, \alpha_2^{\ii}, \dots, \alpha^{\ii}_{l^{\ii}} }) }

\newcommand{\R}{\mathbb{R}}
\newcommand{\dsum}{\displaystyle \sum}
\newcommand{\newi}{\mathbbm{i}}

\newcommand{\card}{\operatorname{card}}
\newcommand{\tr}{^\intercal}

\usepackage{mathrsfs}

\begin{document}

\title{A new characterization of symmetric $H^+$-tensors and $M$-tensors\thanks{This research did not receive any specific grant from funding agencies in the public, commercial, or not-for-profit sectors.}
	}



\author{Xin Shi         \and
        Luis F. Zuluaga 
}


\institute{X. Shi  \at
              Department of Industrial and Systems Engineering, P.C. Rossin College of Engineering \& Applied Science, Lehigh University, PA, 18015 \\
              \email{shi321xin@gmail.com}           
           \and
            L. F. Zuluaga \at
              Department of Industrial and Systems Engineering, P.C. Rossin College of Engineering \& Applied Science, Lehigh University, PA, 18015 \\
\email{luis.zuluaga@lehigh.edu}           
}

\date{Received: date / Accepted: date}

\maketitle

\begin{abstract}
In this work, we present a new characterization of symmetric $H^+$-tensors, also referred as generalized diagonally dominant tensors with nonnegative diagonals. Namely, by exploring their diagonal dominance property, we derive new necessary and sufficient conditions for a symmetric tensor to be an $H^+$-tensor. Based on these conditions, we propose a novel method that allows to identify if a tensor is a symmetric $H^+$-tensor in polynomial time, by solving a power cone optimization problem. Further, we show how this result can be used to efficiently compute the minimum $H$-eigenvalue of symmetric $M$-tensors and to provide tighter lower bounds for the minimum $H$-eigenvalue of the Fan product of two symmetric $M$-tensors. Throughout the article, numerical experiments are used to benchmark and illustrate the applications of our results.
\end{abstract}

\section{Introduction}
\label{sec:introduction}
Tensors can be regarded as a high-order generalization of matrices and they arise in applications in many disciplines of science, engineering and social sciences~\citep[see, e.g.][]{landsberg2012tensors}.
For $m,n \in \mathbb{N}$, an $m$-order $n$-dimensional real tensor is a multidimensional array with the form
\[
\mathcal{A} = (a_{i_1 i_2 \dots i_m}), \quad a_{i_1i_2 \dots i_m} \in \R, \quad 1 \leq i_1 ,i_2,\dots, i_m \leq n.
\]
Matrices are tensors with order $m=2$. Denote $ \mathbb{T}_{m,n}$ as the space of all real tensors with order$~m$ and dimension $n$. Then 
\begin{equation*}
	\mathbb{T}_{m,n} = \underbrace { \R^{n} \otimes \R^{n}  \otimes \cdots \otimes \R^{n} }_{m},
\end{equation*}
where $\otimes $ is the outer product. Denote $[n] = \{1,2,\dots, n\}$. Tensor $\mathcal{A} = (a_{i_1  \dots i_m}) \in  \mathbb{T}_{m,n}$ is called {\em symmetric} if its entries~$a_{i_1  \dots i_m} $ are invariant under any permutation of $(i_1, \dots, i_m)$ for~$i_j \in [n], j \in [m]$. Denote $\mathbb{S}_{m,n}$ as the set of symmetric tensors in $ \mathbb{T}_{m,n}$. The entries~$a_{ii\dots i}$ for any $i \in [n]$ are called {\em diagonal elements (or entries)} of $\mathcal{A}$. 

Following~\citep{cartwright2013number,lim2005singular,qi2005eigenvalues}, for $\mathcal{A} \in \mathbb{T}_{m,n}$, $\lambda \in \mathbb{C}$ is called an {\em eigenvalue} of $\mathcal{A}$, if there exists an {\em eigenvector} $x \in \mathbb{C}^n\backslash \{0 \}$ such that $\mathcal{A}x^{m-1}= \lambda x^{[m-1]}$, where $\mathcal{A}x^{m-1} \in \mathbb{C}^n$ is defined by
\begin{flalign*}
	(\mathcal{A}x^{m-1} )_i =  \sum_{i_2,\dots, i_m = 1}^n a_{ii_2 \dots i_m } x_{i_2}\cdots x_{i_m},
\end{flalign*}
and $x^{[m-1]}\in\mathbb{C}^n \backslash \{0 \}$ is defined by $(x^{[m-1]})_i = x_i^{m-1}$ for all $i \in [n] $. In particular, if $x$ is real, then~$\lambda$ is also real. In this case, we say that $\lambda$ is an {\em H-eigenvalue} of~$\mathcal{A}$.

The {\em comparison tensor} of $\mathcal{A} \in \mathbb{T}_{m,n}$, denoted as $M(\mathcal{A})$, is defined in~\citep{ding2013m,kannan2015some} as follows:
\begin{equation}
	\label{eq:comptensor}
	M(\mathcal{A})_{i_1  \dots i_m} = \begin{cases}
		|a_{i_1  \dots i_m}| & \text{if } i_1 = \cdots = i_m, \\
		- |a_{i_1 \dots i_m}| & \text{otherwise}.
	\end{cases}
\end{equation}
Following~\citep{ding2013m,kannan2015some}, we introduce the next classes of tensors. 
A tensor is called a {\em nonnegative tensor} if all its entries are nonnegative and a tensor is called a {\em diagonal tensor} if all its off-diagonal elements are zero. A tensor~$\mathcal{A}  \in \mathbb{T}_{m,n}$ is said to be a {\em $Z$-tensor} if there exists a nonnegative tensor~$\mathcal{D}  \in \mathbb{T}_{m,n}$ and a nonnegative scalar~$s$ such that $\mathcal{A} = s\mathcal{I} - \mathcal{D}$, where $\mathcal{I}  \in \mathbb{T}_{m,n}$ is a diagonal tensor with all diagonal elements equal to one. For tensor $\mathcal{A}$, denote~$\rho(\mathcal{A})$ as the largest modulus of its eigenvalues. A~$Z$-tensor~$\mathcal{A} = s\mathcal{I} - \mathcal{D}$ is said to be an {\em $M$-tensor} if $s \geq \rho(\mathcal{D})$. If $s > \rho(\mathcal{D})$, then $\mathcal{A}$ is called a {\em strong $M$-tensor}. A tensor is called an {\em $H$-tensor} if its comparison tensor is an $M$-tensor. From the definition of $M$-tensors, the diagonal elements of an $M$-tensor are always nonnegative and the off-diagonal elements are always nonpositive. Thus, the comparison tensor of an $M$-tensor is always itself. If a tensor is an $M$-tensor, then it is also an $H$-tensor with nonnegative diagonal elements. A tensor is called a {\em strong $H$-tensor} if its comparison tensor is a strong $M$-tensor. An $H$-tensor with nonnegative diagonal elements is called an {\em $H^+$-tensor}. The definition of $H^+$-tensors constitutes a natural generalization of the concept of $H^+$-matrices, as introduced in~\citep{boman2005factor} and it matches the definition used in~\citep{luo2015doubly} and~\citep{kannan2015some} where they are referred as $H$-tensors with nonnegative diagonals. Also, as we will discuss later (see Corollary~\ref{coro:gdd_h} and Definition~\ref{def:dd_sdd}), $H$-tensors are equivalent to {\em generalized diagonally dominant tensors}~\citep[][Thm.~4.9]{kannan2015some}. Thus this definition of $H^+$-tensors also matches with the set of generalized diagonally dominant tensors with nonnegative diagonals in \cite{ahmadi2019dsos}. It is worth noting that in~\citep{wang2021accelerated,wang2020global} an alternative definition is used, which does not include $H$-tensors with $0$ elements in the diagonal. However, our results on identifying $H$-tensors with nonnegative diagonals can be straightforwardly applied to detect these tensors as well.

A symmetric tensor is an $H$-tensor if and only if it is a generalized diagonally dominant tensor~\citep[][Thm.~4.9]{kannan2015some}. The matrix version (i.e., when $m$ = 2) of this result is proved in~\citep[][Thm.~8]{boman2005factor} and~\citep{varga1963matrix}.
As a result, it follows that 
a symmetric matrix is an~$H^+$-matrix if and only if it can be written as the sum of a number of positive semidefinite matrices which have a special sparse structure~\citep{boman2005factor}. 
From this fact, it follows
that symmetric~$H^+$-matrices can be identified in polynomial time by solving a {\em second-order cone optimization}~\citep[see, e.g.,][]{lobo1998applications, ahmadi2019dsos}.


$M$-tensors and $H$-tensors have emerged as crucial tools across diverse mathematical and engineering fields, including  hypergraph analysis~\citep{fan2019eigenvectors,sun2019spectral,galuppi2023spectral}, tensor complementarity problems~\citep{huang2017formulating,luo2017sparsest}, multilinear systems~\citep{luo2017sparsest,li2015solving,ding2016solving,wang2019neural}, optimal control problems~\citep{azimzadeh2019high}, high-order Markov chains~\citep{li2014limiting,liu2018tensor}, and as discussed in detail in Example~\ref{example:pop}, have the potential to impact results in polynomial optimization. Next, we provide a brief overview of  these areas of applications.

\begin{itemize}
		\item The Laplacian tensor of a hypergraph is an $M$-tensor. Researchers are actively investigating the spectral properties of hypergraphs by leveraging the properties of $M$-tensors~\citep{sun2019spectral}. Specifically, the chromatic number of a hypergraph is bounded using the largest $H$-eigenvalue of the adjacency tensor~\citep{sun2019spectral,cooper2012spectra}, which can be determined using the methods proposed in this work. The analysis of properties of hypergraphs, such as their chromatic number, arises when modeling problems in areas as varied as 
		informatics, transportation, molecular biology, and telecommunications, to name just a few~\citep[see, e.g.,][]{bretto2013applications,zhang2016radio}. In Example~\ref{example:hypergraph}, we demonstrate how our method can be used to bound the chromatic number of a hypergraph.
			\item Tensor complementarity problems arise in diverse domains, including DNA microarrays, communication systems, and $n$-person non-cooperative games. Research has shown that solutions to tensor complementarity problems involving $M$-tensors and $H$-tensors exhibit desirable properties~\citep{wang2020global,luo2017sparsest}. In Example~\ref{example:tan}, we showcase our method's ability to identify $M$-tensors in a given problem, which allows us to leverage specialized algorithms to efficiently solve  tensor complementarity problems.
	\item While analyzing the existence of solutions for general multilinear systems presents significant challenges, the authors in~\citep{ding2016solving,wang2019existence} have demonstrated the existence of specific solutions for systems involving $M$-tensors and $H$-tensors. Furthermore, numerous efficient solution methods have been developed for such systems~\citep{wang2019neural,wang2020preconditioned}. Efficiently characterizing $M$-tensors and $H$-tensors facilitates the efficient solution of multilinear systems. In Example~\ref{example:pdf}, we illustrate an application of our method to recognize $M$-tensors in a multilinear system, enabling the use of specific algorithms to solve the system.
	\item For high-order Markov chain models, the transition probability tensors are nonnegative tensors. The problem of determining the limiting probability vectors of these tensors can be addressed by solving a nonlinear equation with $M$-tensors. The tensor splitting method proposed in~\citep{liu2018tensor} offers an effective approach to solve such equations. In Example~\ref{example:markov_chain}, we apply our proposed method to two real-world Markov chain models, demonstrating its effectiveness in obtaining their limiting probability vectors with the tensor splitting method proposed in~\citep{liu2018tensor}.
	\item Even order symmetric $H^+$-tensors define globally nonnegative polynomials~\citep{chen2015sos}. 
A recent trend in {\em polynomial optimization}~\citep{lasserre2015introduction} is the derivation of approaches to approximate polynomial optimization problems without the need to use sum of squares polynomials (SOS)~\citep[see, e.g.,][]{kuryatnikova2024reducing,ahmadi2019dsos}. This is mainly motivated by the prohibitively high computational effort needed to 
solve the semidefinite optimization problems associated with the use of SOS polynomials. As detailed in Example~\ref{example:pop}, from the results in~\citep{kuryatnikova2024reducing}, it follows that one can construct (convex) power cone optimization-based hierarchies to approximate any polynomial optimization problem with a compact feasible set. This approximation approach demonstrates the potential impact that our results can have in addressing the solution of practically relevant polynomial optimization problems in fields such as statistics and machine learning, derivative pricing, and control theory~\citep{ahmadi2019dsos}. 
\end{itemize}



In this work we generalize the results on symmetric $H^+$ matrices to symmetric $H^+$-tensors. Namely, we prove that a symmetric tensor is an~$H^+$-tensor if and only if it can be written as the sum of a number of tensors which have a special sparse structure (Theorem~\ref{theo:gdd_decomp}). Based on this, we obtain a novel characterization of symmetric $H^+$-tensors (Theorem~\ref{theo:gdd_characterization})  using {\em conic optimization}~\citep[see, e.g.,][]{wright1997primal} techniques. In particular, we show  
that symmetric~$H^+$-tensors can be identified in polynomial time
(Corollary~\ref{cor:F(A)} and~\eqref{eq:F_A_long})
by solving a {\em power cone optimization}~\citep[see, e.g.,][]{chares2009cones,hien2015differential} problem.

Many efforts have been made to characterize $H$-tensors~\citep[see, e.g.,][]{huang2019iterative,li2014criterions,li2017programmable,liu2017iterative,wang2017new, zhang2016h,zhao2016criterions,sun2020new,huang2019some,liu2020iterative}, by providing sufficient conditions for a tensor to be an $H$-tensor. While these algorithms can identify $H^+$-tensors, there remain $H^+$-tensors that elude detection through these methods. An exception is found in~\citep{luan2019new}, which leverages spectral theory to derive necessary and sufficient conditions for strong $H$-tensors and introduces an iterative algorithm for their identification with linear convergence. On the other hand, our approach 
allows us to take advantage of interior point algorithms for power cone optimization, which operate in polynomial time, and achieve at least a linear convergence rate~\citep{chares2009cones}, to detect symmetric $H^+$-tensors. Furthermore, the sufficient and necessary conditions we present for a symmetric tensor to be an $H^+$-tensor are derived by examining their diagonal dominance properties. This characterization not only facilitates the identification of symmetric $H^+$-tensors but also allows for direct optimization over the set of symmetric $H^+$-tensors. Next, we discuss some advantages of this type of characterization. For that purpose, let us revisit in a bit more detail, one of the applications mentioned earlier.

Consider the problem of computing the minimum $H$-eigenvalue of symmetric $M$-tensors, which plays an important role in a wide range of interesting applications~\citep[see,][and the references therein]{huang2018some}. 
Recent literature~\citep{he2014inequalities,huang2018some,li2013new,tian2010inequalities} focuses on obtaining bounds on the minimum $H$-eigenvalue of $M$-tensors. Our characterization can instead  compute the minimum $H$-eigenvalue of symmetric $M$-tensors in polynomial time by solving a power cone optimization problem (Corollary~\ref{cor:Meigen}). Not surprisingly, the values obtained in this way tighten the bounds provided in~\citep{he2014inequalities,huang2018some,li2013new,tian2010inequalities} (Table~\ref{tab:bounds}). Further, the values are computed in a time faster than the one required to compute the minimum $H$-eigenvalue of $M$-tensors with a more general algorithm~\citep{chen2016computing} that can be used for this purpose (Table~\ref{tab:comparison2}).
To illustrate the practical applications of these results, we show (Example~\ref{example:hypergraph}) how to obtain an upper bound on the chromatic number of a {\em hypergraph}~\citep[see, e.g.,][]{cooper2012spectra} by computing the minimum $H$-eigenvalue of its associated transformed {\em adjacency tensor}~\citep[see, e.g.,][]{chang2013survey}. 
As additional applications, in Examples~\ref{example:tan}, \ref{example:pdf}, and \ref{example:markov_chain}, we demonstrate how computing the minimum $H$-eigenvalues of $Z$-tensors enables us to determine whether specialized methods can be applied to find the sparsest solution of a tensor complementarity problem, obtain limiting probability vectors of high-order Markov chains, or, more generally, solve multilinear systems of equations.

Further, 
consider the problem of finding  the minimum $H$-eigenvalue of the {\em Fan product}~\citep{fan1964inequalities} of two symmetric $M$-tensors. 
One of the main characteristics of this product is that the Fan product of $M$-tensors is also an $M$-tensor~\citep{shen2019some}. Some bounds for the minimum $H$-eigenvalue of the Fan product of $Z$-matrices ($Z$-tensors) are proposed in~\citep{cheng2014new,fang2007bounds,shen2019some}. 
Our characterization can be used to obtain bounds that are theoretically and empirically tighter than any of the bounds provided in~\citep{shen2019some} (Table \ref{tab:fan_prod}).

The remaining of the article is organized as follows: Section \ref{part:pre} introduces additional notation, definitions and some basic results. In Section \ref{part:dd_sdd}, the characterizations of symmetric~$H^+$-tensors are presented. With these characterizations, we provide a way to identify if a tensor is a symmetric $H^+$-tensor in polynomial time. In Section \ref{sec:Meigenvalue}, we show how to obtain the minimum $H$-eigenvalue of a symmetric $M$-tensor by applying the methodology proposed in this work. We provide some applications of our characterizations of symmetric $H^+$-tensors and $M$-tensors in this chapter. In Section \ref{section:fan_prod}, we further apply our results to obtain lower bounds for the minimum $H$-eigenvalue of the Fan product of two symmetric $M$-tensors, that are tighter than the ones proposed in the related literature. 
Section \ref{part:conclusion} concludes the article with some final remarks.

All the computational experiments mentioned in this work were implemented in MATLAB
R2022b using the \texttt{Systems Polynomial Optimization Toolbox} (\texttt{SPOT})~\citep{SPOT}, and the solver \texttt{MOSEK 9.3.22}~\citep{mosek2022}, using an Intel computer Core i7-4770HQ with 2.20 GHz frequency and 16 GB RAM memory. The packages \texttt{
	allcomb(varargin) v4.2}~\citep{jos_allcomb_2025} and \texttt{Tensor Toolbox for MATLAB v3.1}~\citep{bader_tensor_toolbox_2019} are also employed to formulate the power cone optimization problems. 
The Github repository \href{https://github.com/XinEDprob/spotless}{https://github.com/XinEDprob/spotless} makes publicly available
all the data and code used to generate
the computational results presented in the article.


\section{Preliminaries}
\label{part:pre}

For ease of exposition, in what follows, we use small letters $a,b,\dots$ for scalars and vectors; capital letters $A,B, \dots$ for matrices; calligraphic letters~$\mathcal{A}, \mathcal{B}, \dots$ for tensors and~$\mathscr{A}, \mathscr{B},\dots $ for index sets; and blackboard bold letters $\mathbb{T}, \mathbb{D}, \dots$ for other kinds of sets or spaces in this work.

First we introduce additional notation and some fundamental properties of tensors. Let $\R[x]: = \R [ x_1,\dots, x_n ] $ be the set of polynomials in $n$ variables with real coefficients. A polynomial $p\in \R[x]$ is called a sum of squares (SOS) if it can be written as $p = \sum_{i} q_i^2$ for a finite number of polynomials $q_i \in \R[x]$. Tensor $\mathcal{A} =  (a_{i_1 i_2 \dots i_m}) \in \mathbb{S}_{m,n}$ is said to have an SOS-tensor decomposition if its corresponding polynomial 
\begin{equation}
	\label{eq:tensor_equation}
	{\mathcal{A}}x^m =  \sum_{i_1, i_2,\dots, i_m = 1}^n a_{i_1 i_2 \dots i_m } x_{i_1} x_{i_2}\cdots x_{i_m}
\end{equation}
is an SOS~\citep[see, e.g.,][]{luo2015linear}. A tensor is called a PSD tensor if its corresponding polynomial is globally nonnegative~\citep[see, e.g.,][]{luo2015linear}. The authors in~\citep{chen2015sos} show that every even order symmetric $H^+$-tensor has an SOS-tensor decomposition. 

\begin{theorem}[{\citep[][Thm. 3.7]{chen2015sos}}]
	\label{theo:H_plus_SOS}
	Let $m, n \in \mathbb{N}$ and $\mathcal{A} \in \mathbb{S}_{m,n}$ be an $H^+$-tensor. If $m$ is even, then $\mathcal{A}$ has an SOS-tensor decomposition.
\end{theorem}

From Theorem \ref{theo:H_plus_SOS}, it follows that an even order symmetric $H^+$-tensor is also a PSD tensor. On the other hand, symmetric $H^+$-tensors can be characterized using the notion of diagonally dominant tensors (see Definition \ref{def:dd_sdd}). Most of the work related to $H^+$-tensors makes use of the diagonal dominance property~\citep[see, e.g.,][]{huang2019iterative,li2014criterions,li2017programmable,wang2017new, zhao2016criterions}. We will also make use of this property in our results; hence, we present some related definitions.

\begin{definition}[{\citep[][Def. 6.5]{luo2016completely}}]
	\label{def:dd_sdd}
	Let $m, n \in \mathbb{N}$ and $ \mathcal{A}= (a_{i_{1} \dots i_m}) \in \mathbb{T}_{m,n}$.
	\begin{enumerate}[label=(\roman*)]
		\item $\mathcal{A}$ is called a diagonally dominant (DD) tensor if 
		\begin{equation}
			\label{inner1:diag}
			|a_{ii\dots i}| \geq \sum_{(i_2,\dots,i_m) \neq (i,\dots,i)} |a_{i i_2 \dots i_m}|, \ \ \forall \  i \in [n].
		\end{equation}
		\item $\mathcal{A}$ is called a generalized diagonally dominant (GDD) tensor if there exists a positive diagonal matrix $D$ such that the tensor $\mathcal{A} D^{1-m}D\cdots D$ defined as 
		\begin{equation}
			\label{eq:gdd_def1}
			(\mathcal{A}D^{1-m}D\cdots D)_{i_1\dots i_m} = a_{i_1 \dots i_m} d_{i_1}^{1-m} d_{i_2}\cdots d_{i_m}, \ \ \forall  i_1, \dots, i_m \in [n],
		\end{equation}
		is diagonally dominant, where $d_{i} = D_{ii}$ is the $i$th diagonal element of $D$.
		\label{defitemii}
	\end{enumerate}
\end{definition}

From the definition of DD tensors and GDD tensors, one can derive 
the following equivalent definition of GDD tensors.

\begin{proposition}
	\label{prop: alter_gdd_def}
	Let $m, n \in \mathbb{N}$, then $\mathcal{A} \in \mathbb{T}_{m,n}$ is a GDD tensor if and only if there exists a positive diagonal matrix $D$ such that the tensor $\mathcal{A} DD\dots D$ defined as 
	\begin{equation}
		\label{eq:gdd_def2}
		(\mathcal{A}DD\cdots D)_{i_1\dots i_m} = a_{i_1 \dots i_m} d_{i_1}d_{i_2}\cdots d_{i_m}, \ \ \forall  i_1,  \dots, i_m \in [n],
	\end{equation}
	is diagonally dominant, where $d_{i} = D_{ii}$ is the $i$th diagonal element of $D$. If~$\mathcal{A} \in \mathbb{S}_{m,n}$, then $\mathcal{A} DD\cdots D \in \mathbb{S}_{m,n}$.
\end{proposition}
\begin{proof}
	From Definition~\ref{def:dd_sdd}\ref{defitemii}, if $\mathcal{A} = (a_{i_{1} \dots i_m}) \in \mathbb{T}_{m,n}$ is a GDD tensor, then there exists a positive diagonal matrix $D $ such that $\mathcal{A}D^{1-m}D\cdots D$ is a DD tensor. That is for all~$i \in [n]$,
	\begin{equation}
		\label{eq:prop1_prof1}
		|(\mathcal{A}D^{1-m}D\cdots D)_{i\dots i}| \geq \sum_{(i_2,\dots,i_m) \neq (i,\dots,i)}  |(\mathcal{A}D^{1-m}D\cdots D)_{i i_2\dots i_m}| .
	\end{equation}
	Note that \eqref{eq:prop1_prof1} is equivalent to 
	\begin{equation}
		\label{eq:prop1_prof2}
		|a_{i\dots i}| \geq \sum_{(i_2,\dots,i_m) \neq (i,\dots,i)}  |a_{i \dots i_m} d_{i}^{1-m} d_{i_2}\cdots d_{i_m}|.
	\end{equation}
	Considering that $d_{i} >0 $ for all $i \in [n]$, and multiplying by $d_i^m$ on both sides of~\eqref{eq:prop1_prof2}, we have that
	\begin{equation}
		\label{eq:prop1_prof3}
		|a_{i\dots i} | d_i^m \geq \sum_{(i_2,\dots,i_m) \neq (i,\dots,i)}  |a_{i \dots i_m} | d_{i} d_{i_2}\cdots d_{i_m}, 
	\end{equation}
	for all $ i \in [n]$. Thus, the tensor $\mathcal{A} DD\dots D$ defined by \eqref{eq:gdd_def2} is a DD tensor.
	
	For the another direction, if the tensor $\mathcal{A} DD\cdots D$ defined by \eqref{eq:gdd_def2} is a DD tensor for a positive diagonal matrix~$D$, then inequality \eqref{eq:prop1_prof3} holds for all $i \in [n]$. Dividing both sides of \eqref{eq:prop1_prof3} by $d_i^m > 0$,  we have inequality \eqref{eq:prop1_prof2}, which is equivalent to \eqref{eq:prop1_prof1}, 	for all $ i \in [n]$ and shows that $\mathcal{A}$ is a GDD tensor.
\end{proof}

For the remainder of the article, denote by ${DD}_{m,n}$ and ${GDD}_{m,n}$ the set of DD tensors and the set of GDD  tensors in $\mathbb{S}_{m,n}$, respectively. DD and GDD tensors with nonnegative diagonal elements will be referred as DD$^+$ and GDD$^+$ tensors, respectively. Also, denote by $DD^+_{m,n}$ and $GDD^+_{m,n}$ the set of DD$^+$ tensors and the set of GDD$^+$ tensors in $\mathbb{S}_{m,n}$, respectively. The set of PSD tensors in $\mathbb{S}_{m,n}$ is denoted as $PSD_{m,n}$.

For $n \in \mathbb{N}$, a set $\mathbb{W} \subset \R^n$ is called a cone if $0 \in \mathbb{W}$ and~$x \in \mathbb{W}$ implies~$\lambda x \in \mathbb{W}$ for any $\lambda \geq 0 $. A set $\mathbb{W}$ is called a convex cone if $\lambda x + \mu y \in \mathbb{W}$ for any $x,y \in \mathbb{W}$ and any $\lambda, \mu \geq 0$. Given a set $\mathbb{W}$, let $cone(\mathbb{W}) = \{ \lambda x \ | \ x \in \mathbb{W}, \lambda \geq 0 \}$ be the {\em conic hull} of $\mathbb{W}$; and $convex(\mathbb{W}) = \{ \lambda x + \mu y \ | \ x, y \in \mathbb{W}, \lambda ,\mu \geq  0, \lambda + \mu = 1 \}$ be the {\em convex hull} of $\mathbb{W}$.

Clearly, for $m,n \in \mathbb{N}$, $DD^+_{m,n}$ is a convex cone. We will show that~$GDD^+_{m,n}$ is also a convex cone later (see Proposition \ref{prop:convex_cone_GDD}). Next we present a characterization of symmetric $H$-tensors via symmetric GDD tensors.

\begin{theorem}[{\citep[][Thm.~4.9]{kannan2015some}}]
	\label{them:gdd_h}
	Let $m, n \in \mathbb{N}$ and $\mathcal{A} \in \mathbb{S}_{m,n}$. Then $\mathcal{A}$ is an~$H$-tensor if and only if $\mathcal{A} \in GDD_{m,n}$.
\end{theorem}

\begin{corollary}
	\label{coro:gdd_h}
	Let $m, n \in \mathbb{N}$ and $\mathcal{A} \in \mathbb{S}_{m,n}$. Then $\mathcal{A}$ is an $H^+$-tensor if and only if $\mathcal{A} \in GDD^+_{m,n}$.
\end{corollary}

From Theorem \ref{theo:H_plus_SOS} and Corollary \ref{coro:gdd_h}, if $m$ is even, we have the following inclusion relationships:
\begin{equation*}
	DD^+_{m,n} \subseteq GDD^+_{m,n} \subseteq PSD_{m,n}.
\end{equation*}
In light of Corollary \ref{coro:gdd_h}, in what follows, we will take the liberty to use both symmetric $H^+$ and GDD$^+$ interchangeably to refer to symmetric $H^+$-tensors. 

Denote $\card(A)$ as the cardinality of the set $A$. For $m, n \in \mathbb{N}$, define the index sets

\begin{equation*}
	\begin{split}
		\mathscr{D}_{n}^m = \{(i_1, \dots, i_m) \ | \ 1 \leq i_1 \leq \cdots \leq i_m \leq n \} \cap \\
		\{(i_1, \dots , i_m) \ | \ \card(\{i_1, \dots, i_m\}) > 1\} ,
	\end{split}
\end{equation*}

and
\[
\mathscr{F}_{n}^m = \{ (\underbrace{i,i,\dots, i}_{m}) \ | \ i \in [n] \}.
\]

For any index $(i_1,\dots, i_m) \in \mathscr{D}_{n}^m \cup \mathscr{F}_{n}^m$,  denote $\mathscr{P}_{i_1  \dots i_m} $ as the set of all permutations of~$i_1, \dots, i_m$ and denote 
\[
\mathscr{Q}_{i_1  \dots i_m} = \{(\underbrace{p,p, \dots, p}_{m}) \ | \ p \in \{ i_1, \dots i_m  \}\}.\]
Also, for $(i_1,\dots,i_m) \in \mathscr{D}_{n}^m \cup \mathscr{F}_{n}^m$, let $\mathbb{D}_{m,n}^{i_1  \dots i_m} \in  \mathbb{S}_{m,n}$ be the set of sparse tensors defined as follows:

\begin{equation}
	\label{eq:def_D_mn_i}
	\begin{split}
		\mathbb{D}_{m,n}^{i_1 \dots i_m} = \{ (a_{j_1 \dots j_m}) \in \mathbb{S}_{m,n} \ | \ a_{j_1 \dots j_m} = 0 \text{ \ if \ } \\
		(j_1, \dots, j_m) \notin \mathscr{P}_{i_1 \dots i_m} \cup \mathscr{Q}_{i_1 \dots i_m} \}.
	\end{split}
\end{equation}

Further, let
\[
\mathbb{D}_{m,n} =\bigcup_{(i_1,\dots, i_m) \in \mathscr{D}_n^m} \mathbb{D}_{m,n}^{i_1  \dots i_m}.
\]	

To proceed, we introduce the following class of tensors. 
\begin{definition}
	\label{def:two_subclass}
	For $m, n \in \mathbb{N}$ and any $(i_1, \dots, i_m) \in \mathscr{D}_{n}^m $, $c \in \{0 ,1\}$, denote $\mathcal{V}^{c, i_1 \dots i_m} = (v^{c, i_1  \dots i_m}_{j_1  \dots j_m}) \in \mathbb{D}_{m,n}^{i_1  \dots i_m}$,  as the tensor satisfying:
	\begin{enumerate}[label=(\roman*)]
		\item $v^{c, i_1   \dots i_m  }_{j_1\dots j_m} = (-1)^c$ if $(j_1,\dots ,j_m) \in \mathscr{P}_{i_1 \dots i_m}$.
		\item  The value of $j$-th diagonal element is equal to the sum of the absolute values of the off-diagonal entries on the $j$-th slice (the diagonal elements are excluded in the sum); that is \\
		\[v_{j j  \dots j  }^{c,i_1 \dots i_m} = \sum_{ (j_2, \dots, j_m ) \neq (j, \dots, j ) } |v_{j j_2  \dots j_m  }^{c,i_1  \dots i_m}|, \forall  \ j\in [n].\]
	\end{enumerate}
	Further, for all $i \in [n]$, denote $\mathcal{V}^{0,i i \dots i}$ as the tensor where the only nonzero entry is $v^{0,ii\dots i}_{ii\dots i} = 1$; and $\mathcal{V}^{1,i i \dots i}$ as the tensor with all entries set to 0. Also, let $\mathbb{E}_{m,n} = \{  \mathcal{V}^{c, i_1 \dots i_m}   \ | \ c \in \{0, 1\},  (i_1, \dots, i_m) \in \mathscr{D}_{n}^m \cup \mathscr{F}_{n}^m   \}  $.
\end{definition}

From Definition \ref{def:two_subclass}, it follows that for all $(i_1, \dots, i_m) \in \mathscr{D}_{n}^m \cup \mathscr{F}_{n}^m$ and~$c \in \{0,1\}$, $\mathcal{V}^{c,i_1 \dots i_m} \in DD_{m,n}^+$. For example, when $m = 2$ and $n = 4$, we have
\begin{equation*}
	\mathcal{V}^{0,12} = \begin{bmatrix}
		1    & 1 & 0 & 0 \\
		1       & 1 & 0 &0 \\
		0  &  0& 0 & 0\\
		0  &  0& 0 &0
	\end{bmatrix}, \quad \quad 
	\mathcal{V}^{1,13} = \begin{bmatrix}
		1    & 0 & -1 & 0 \\
		0       & 0 & 0 &0 \\
		-1  &  0& 1 & 0\\
		0  &  0& 0 &0
	\end{bmatrix}.
	\quad \quad 
\end{equation*}

For ease of exposition, we also introduce an auxiliary notation for indices. For $m, n \in \mathbb{N}$, index $ \ii  :=(i_1, i_2,\dots, i_m) \in  \mathscr{D}_n^m $ and some $l^{\ii} \in[m]$, we call 
\[
(\tightindex, \tightpower) \in { [n]^{l^{\ii}} \times [m]^{l^{\ii}}}
\]
as the {\em tight pair} of $\ii$ if $\tightindex$ and $\tightpower$ satisfy 
\begin{equation}
	\label{eq:tightpair}
	x_{i_1} x_{i_2} \dots x_{i_m} = x_{j_1^{\ii}}^{\alpha_1^{\ii}} x_{j_2^{\ii}}^{\alpha_2^{\ii}} \dots x_{j_{l^{\ii}}^{\ii}}^{\alpha_{l^{\ii}}^{\ii}},
\end{equation}
where $ 1 \leq j_1^{\ii} < j_2^{\ii} < \dots < j_{l^{\ii}}^{\ii} \leq n$. We will refer to~$\tightindex$ as the {\em tight index} and to $\tightpower$ as the tight power.

The following example illustrates the tight pair notation and shows the benefit of introducing it. 

\begin{example}
	\label{ex:tightpair} 
	Assume index~$ \ii   = (i_1, i_2,i_3, i_4, i_5, i_6) = (1, 1, 1, 2, 2, 3) \in \mathscr{D}_3^6$, then the tight index of $\ii$ is $(j_1^{\ii}, j_2^{\ii}, j_3^{\ii}) = (1, 2, 3)$ and the tight power of $\ii$  is $( \alpha_1^{\ii}, \alpha_2^{\ii}, \alpha_2^{\ii}   ) = (3,2,1)$ as 
	\[
	x_{i_1} x_{i_2} x_{i_3} x_{i_4} x_{i_5} x_{i_6} = x_1^3 x_2^2 x_3.
	\] 
	
	Let $\mathcal{A} = (a_{i_1 i_2 i_3 i_4 i_5 i_6})\in \mathbb{S}_{6,3}$, one can easily obtain that the coefficient of $x_1^3 x_2^2 x_3$ in the corresponding polynomial $\mathcal{A} x^m$ is 
	\[
	\binom{m}{\alpha_1^{\ii}, \alpha_2^{\ii}, \alpha_3^{\ii}} a_{111223}= \binom{6}{3, 2, 1} a_{111223}.
	\]
	Besides, it is clear that the element $a_{111223}$ in $\mathcal{A}$ is multiplied by $x_1^3 x_2^2 x_3$ in the polynomial $\mathcal{A} x^m$.
\end{example}

It is important to note, however, that we will routinely drop the upper index $\vec{i}$ in the notation when the $\vec{i}$ we are referring to is clear from (or fixed in) the context. Further, denote $e_j $ as the unitary vector in the $j$th direction of appropriate dimensions.

\section{New characterization of symmetric $H^+$-tensors}
\label{part:dd_sdd}

In this section, we present a new characterization of symmetric $H^+$-tensors, or equivalently GDD$^+$ tensors (cf., Corollary \ref{coro:gdd_h}),  based on the {\em power cone}~\citep{chares2009cones,hien2015differential}. First, we characterize the set of DD$^+$ tensors.

\begin{proposition}
	\label{prop:DDT_cone}
	For $m, n \in \mathbb{N}$, $DD^+_{m,n} = convex(cone(\mathbb{E}_{m,n}))$ and each tensor in $\mathbb{E}_{m,n}$ generates an extreme ray of $ DD^+_{m,n} $.
\end{proposition}

\begin{proof}
	First, from Definition \ref{def:two_subclass}, it follows that $\mathbb{E}_{m,n} \subseteq DD^+_{m,n}$. This, together with the fact that $DD_{m,n}^+$ is a convex cone, implies that $convex(cone(\mathbb{E}_{m,n})) \subseteq DD^+_{m,n} $.
	
	Second, for  $\mathcal{A} = (a_{i_1\dots i_m}) \in DD^+_{m,n}$, denote $$  \mathscr{P}_+ =  \{(i_1, \dots, i_m) \in \mathscr{D}_n^m \ | \   a_{i_1i_2\dots i_m} \geq 0  \}$$ and $$ \mathscr{P}_- =  \{(i_1,\dots, i_m)\in \mathscr{D}_n^m \ | \  a_{i_1i_2\dots i_m} < 0  \}.$$ Then 
	\begin{flalign}
		\label{eq:DD_cone}
		\mathcal{A} = &\sum_{i = 1}^n \left(a_{i i \dots i} - \sum_{(i_2,\dots,i_m) \neq (i,\dots,i)} |a_{i i_2 \dots i_m}|\right) \mathcal{V}^{0, i i \dots i} \\
		+ &\sum_{ (i_1,i_2,\dots, i_m) \in \mathscr{P}_+ } a_{i_1 i_2 \dots  i_m} \mathcal{V}^{0, i_1i_2 \dots i_m}  +  \sum_{ (i_1,i_2,\dots, i_m) \in \mathscr{P}_- }( - a_{i_1 i_2 \dots  i_m}) \mathcal{V}^{1, i_1i_2 \dots i_m} \nonumber.
	\end{flalign}
	
	Since $\mathcal{A} \in DD_{m,n}^+$, $a_{i i \dots i} \geq  \sum_{(i_2,\dots,i_m) \neq (i,\dots,i)} |a_{i i_2 \dots i_m}|$ for all $i \in [n]$. Thus, $\mathcal{A} $ is in the convex hull of the conic hull of $\mathbb{E}_{m,n}$, after noticing that all the coefficients in the right hand side of \eqref{eq:DD_cone} are nonnegative. That is $DD^+_{m,n} \subseteq convex(cone(\mathbb{E}_{m,n}))$.
\end{proof}

To give a similar characterization for GDD$^+$ tensors, we need the following results first.

\begin{theorem}[{\citep[][Thm.~1(a)]{qi2013symmetric}}]
	\label{lemma:nonnegative_tensor_pho}
	For $m, n \in \mathbb{N}$, if $\mathcal{D} \in \mathbb{S}_{m,n}$ is a nonnegative tensor, then $\rho(\mathcal{D})$ is an $H$-eigenvalue of $\mathcal{D}$. 
\end{theorem}

Denote the largest $H$-eigenvalue of tensor $\mathcal{A} \in \mathbb{S}_{m,n}$ as $\lambda_{\max}(\mathcal{A})$. 
\begin{theorem}[{\citep[][Thm.~2]{qi2013symmetric}}]
	\label{lemma:nonnegative_tensor_max}
	For $m, n \in \mathbb{N}$, if $\mathcal{A} \in \mathbb{S}_{m,n}$ is a nonnegative tensor, then
	\[
	\lambda_{\max}(\mathcal{A}) = {\max} \left\{  \mathcal{A} x^m: x\in \R^n_+, \sum_{i=1}^n x_i^m = 1 \right\}.
	\]
\end{theorem}

Using an approach similar to the one used to prove {\citep[][Thm.~4.5]{zhang2014m}}, we can establish a slight generalization of that result in Proposition~\ref{prop:nonnegative_tensor_sum}.
\begin{proposition}
	\label{prop:nonnegative_tensor_sum}
	For $m, n \in \mathbb{N}$, if both $\mathcal{A} \in \mathbb{S}_{m,n}$ and $\mathcal{B} \in \mathbb{S}_{m,n}$ are nonnegative tensors, then $\rho(\mathcal{A} + \mathcal{B}) \leq \rho(\mathcal{A}) + \rho(\mathcal{B})$.
\end{proposition}
\begin{proof} 
	Let $\mathcal{D} \in \mathbb{T}_{m,n}$. From the definition of $\rho(\mathcal{D})$ and $\lambda_{\max}(\mathcal{D})$, it clearly follows that $\rho(\mathcal{D}) \geq \lambda_{\max}(\mathcal{D})$. If $\mathcal{D}$ is a symmetric nonnegative tensor, it then follows from Theorem \ref{lemma:nonnegative_tensor_pho} that 
	\begin{equation}
		\label{eq:rho_inequality_pf}
		\rho(\mathcal{D}) = \lambda_{\max}(\mathcal{D}).
	\end{equation}
	Let $\mathcal{A} \in \mathbb{S}_{m,n}$, and $\mathcal{B} \in \mathbb{S}_{m,n}$ be nonnegative tensors. Then we have from equation \eqref{eq:rho_inequality_pf} that $\rho(\mathcal{A}) = \lambda_{\max}(\mathcal{A})$ and $\rho(\mathcal{B}) = \lambda_{\max}(\mathcal{B})$. Furthermore, it follows from Theorem \ref{lemma:nonnegative_tensor_max} that
	\begin{align*}
		\lambda_{\max}(\mathcal{A} + \mathcal{B}) &= {\max} \left\{  (\mathcal{A} + \mathcal{B}) x^m: x\in \R^n_+, \sum_{i=1}^n x_i^m = 1 \right\} \\ 
		& = {\max} \left\{  \mathcal{A} x^m + \mathcal{B} y^m : x,y\in \R^n_+, \sum_{i=1}^n x_i^m = 1, \sum_{i=1}^n y_i^m = 1, x = y\right\} \\ 
		& \leq {\max} \left\{  \mathcal{A} x^m: x\in \R^n_+, \sum_{i=1}^n x_i^m = 1\right\}  \\
		&\quad + {\max} \left\{  \mathcal{B} y^m: y\in \R^n_+, \sum_{i=1}^n y_i^m = 1\right\} \\
		&	= \lambda_{\max}(\mathcal{A} ) + \lambda_{\max}(\mathcal{B}).
	\end{align*}
	
	To finish, notice that $\mathcal{A} + \mathcal{B}$ is a symmetric nonnegative tensor. Thus after using equation \eqref{eq:rho_inequality_pf} for the tensor $\mathcal{A} + \mathcal{B}$, we conclude that $\rho(\mathcal{A} + \mathcal{B}) = \lambda_{\max}(\mathcal{A} + \mathcal{B}) \leq  \lambda_{\max}(\mathcal{A}) +  \lambda_{\max}(\mathcal{B}) = \rho(\mathcal{A}) + \rho(\mathcal{B}) $.
\end{proof}

\begin{proposition}[{\citep[][Prop.~2.7]{kannan2015some}}]
	\label{prop:z_tensor_m_tensor}
	For $m, n \in \mathbb{N}$, let $\mathcal{B} \in \mathbb{S}_{m,n}$ be a $Z$-tensor such that $\mathcal{A} \leq \mathcal{B}$ where $\mathcal{A}$ is an $M$-tensor. Then $\mathcal{B}$ is also an $M$-tensor.
\end{proposition}

\begin{proposition}
	\label{prop:convex_cone_GDD}
	For $m, n \in \mathbb{N}$, $GDD_{m,n}^+$ is a convex cone.
\end{proposition}
\begin{proof}
	Let $\mathcal{A} = (a_{i_1\dots i_m})\in GDD_{m,n}^+  $ and $\mathcal{B}  = (b_{i_1\dots i_m})\in  GDD_{m,n}^+  $. From Corollary~\ref{coro:gdd_h}, both $\mathcal{A}$ and $\mathcal{B}$ are symmetric $H^+$-tensors. Thus $M(\mathcal{A})$ and $M(\mathcal{B})$ are symmetric $M$-tensors. That is, there exist nonnegative scalars $s_1, s_2$ and nonnegative tensors $\mathcal{D}_1$ and~$\mathcal{D}_2$ such that $M(\mathcal{A}) = s_1 I - \mathcal{D}_1$, $M(\mathcal{B}) = s_2 I - \mathcal{D}_2$ and $s_1 \geq \rho(\mathcal{D}_1)$, $s_2 \geq \rho(\mathcal{D}_2)$. Then~$M(\mathcal{A}) + M(\mathcal{B}) = (s_1 + s_2)I - (\mathcal{D}_1 + \mathcal{D}_2)$. Since $s_1 + s_2 \geq 0$ and $D_1 + D_2$ is a nonnegative tensor, $M(\mathcal{A}) + M(\mathcal{B})$ is a symmetric $Z$-tensor. Also, from Proposition~\ref{prop:nonnegative_tensor_sum}, if follows that~$\rho(\mathcal{D}_1 + \mathcal{D}_2) \leq \rho(\mathcal{D}_1 ) + \rho(\mathcal{D}_2) \leq s_1 + s_2$. Thus, $M(\mathcal{A}) + M(\mathcal{B})$ is also a symmetric $M$-tensor.

	Next, we prove that $M(\mathcal{A}+\mathcal{B})$ is a $Z$-tensor. Recall that $M(\mathcal{A}+\mathcal{B})$ is the comparison matrix of $\mathcal{A}+\mathcal{B}$. Thus, all its diagonal elements are nonnegative and all off-diagonal elements are nonpositive. Denote $s = { \max}\{ |a_{ii\dots i}| + |b_{ii\dots i}| ,  i\in[n]  \}$. Then $M(\mathcal{A}+\mathcal{B}) = s I - (sI -M(\mathcal{A}+\mathcal{B}) )$ where $sI -M(\mathcal{A}+\mathcal{B})$ is a nonnegative tensor. Thus, $M(\mathcal{A}+\mathcal{B}) $ is a $Z$-tensor.

	From the definition of comparison tensors and the fact that $\mathcal{A}, \mathcal{B}$ have nonnegative diagonal elements, $M(\mathcal{A}+\mathcal{B}) \geq M(\mathcal{A}) + M(\mathcal{B})$ componentwise. From the fact that $M(\mathcal{A}) + M(\mathcal{B})$ is an $M $-tensor and $M(\mathcal{A}+\mathcal{B})$ is a $Z$-tensor, it follows from Proposition~\ref{prop:z_tensor_m_tensor} that $M(\mathcal{A}+\mathcal{B})$ is also an $M$-tensor. Thus $\mathcal{A}+\mathcal{B}$ is a symmetric $H^+ $-tensor, and from Corollary \ref{coro:gdd_h}, $\mathcal{A}+\mathcal{B}$ is a GDD$^+$ tensor. Thus, $\mathcal{A} + \mathcal{B} \in GDD_{m,n}^+$. This, together with the fact that $\mathcal{A} \in GDD_{m,n}^+$ implies~$\lambda \mathcal{A} \in GDD_{m,n}^+$ for any nonnegative scalar $\lambda$, implies that $GDD_{m,n}^+$ is a convex cone.
	
\end{proof}

\begin{theorem}
	\label{theo:gdd_decomp}
	For $m, n \in \mathbb{N}$, $\mathcal{A} \in  GDD_{m,n}^+$ if and only if $\mathcal{A} = \sum_{i =1}^{r} B_i$ where $r \in \mathbb{N}$ and $B_i \in \mathbb{D}_{m,n}  \cap GDD_{m,n}^+$.
\end{theorem}
\begin{proof}
	For $m, n \in \mathbb{N}$, let $\mathcal{A} \in GDD^+_{m,n}$. Then, from Proposition \ref{prop: alter_gdd_def}, there exists a positive diagonal matrix $D$ such that $\mathcal{B}: = \mathcal{A} DD \cdots D \in DD^+_{m,n}$. From Proposition~\ref{prop:DDT_cone}, it follows that there exist $r \in \mathbb{N}$, $\lambda_i \geq 0$, $\mathcal{C}_i \in \mathbb{E}_{m,n} \subset \mathbb{D}_{m,n} \cap DD_{m,n}^+$ for~$ i \in [r]$ such that $\mathcal{B} = \sum_{i=1}^{r} \lambda_i \mathcal{C}_i$. Then $\mathcal{A} = \sum_{i=1}^{r} \lambda_i \mathcal{C}_i D^{-1} \cdots D^{-1} D^{-1}$. Let $\mathcal{B}_i =\lambda_i \mathcal{C}_i D^{-1} \cdots D^{-1} D^{-1}$ for all $ i \in [r]$. Then the {only if} statement follows after noticing that for all $ i \in [r]$, $\mathcal{B}_i \in GDD_{m,n}^+$ and $\mathcal{B}_i \in \mathbb{D}_{m,n}$ (as multiplying with positive numbers will not affect the sparse structure of tensors $\mathcal{C}_i \in \mathbb{D}_{m,n}$, $i \in [r]$). For the if statement, note that if $\mathcal{A} = \sum_{i=1}^r \mathcal{B}_i$ with $\mathcal{B}_i \in \mathbb{D}_{m,n}^+ \cap GDD_{m,n}^+$ for all $i \in [r]$, then, from Proposition \ref{prop:convex_cone_GDD}, we have $\mathcal{A} \in GDD_{m,n}^+$.
\end{proof}

The matrix version (i.e. $m=2$) of Theorem \ref{theo:gdd_decomp} was presented in \citep{ahmadi2019dsos,boman2005factor}.

\begin{lemma}[{\citep[][Lem.~3.8]{ahmadi2019dsos}}]
	\label{lemma:gdd_matrix_decomp}
	For $n \in \mathbb{N}$, if matrix $A \in \mathbb{S}_{2,n} $, then $A $ is a GDD$^+$ matrix if and only if $A = \sum_{i<j}  M^{ij}$ where each $M^{ij} \in \mathbb{S}_{2,n}$ with zeros everywhere except for four entries $(M^{ij})_{ii}$, $(M^{ij})_{ij}$, $(M^{ij})_{ji}$, $(M^{ij})_{jj}$ which make $M^{ij}$ symmetric and positive semidefinite. 
\end{lemma}

It is easy to see that $M^{ij}$ in Lemma \ref{lemma:gdd_matrix_decomp} is positive semidefinite if and only if $M^{ij}$ is a GDD$^+$ matrix. Thus, Lemma~\ref{lemma:gdd_matrix_decomp} can be regarded as a special case of Theorem~\ref{theo:gdd_decomp}. In Theorem \ref{theo:gdd_characterization}, we provide sufficient and necessary conditions for a tensor to be in $\mathbb{D}_{m,n} \cap  GDD_{m,n}^+ $ (i.e., a sparse GDD$^+$ tensor).

\begin{theorem} 
	\label{theo:gdd_characterization}
	Let $m, n \in \mathbb{N}$, $(i_1, \dots, i_m) \in \mathscr{D}_{n}^m \cup  \mathscr{F}_{n}^m$, and a tensor $\mathcal{B} = (b_{p_1  \dots p_m}) \in \mathbb{D}^{i_1   \dots i_m}_{m,n}  $ be given. Then,
	\begin{enumerate}[label=(\roman*)]
		\item if $(i_1, \dots, i_m) \in \mathscr{D}_{n}^m $, $\mathcal{B} \in GDD^+_{m,n}$ if and only if its entries satisfy 
		\begin{equation}
			\label{eq:geo_mean_cone_}
			\prod_{k=1}^l b_{j_k j_k\dots j_k}^{\alpha_k}   \geq c| b_{i_1 \dots i_m} |^m,
		\end{equation}
		where $c = \prod_{k=1}^l {\binom{m-1}{\alpha - e_k} }^{\alpha_k}$, and $((j_1, \dots, j_l), \alpha = (\alpha_1,   \dots. \alpha_l))$ is the tight pair associated with $(i_1,\dots, i_m)$, and 
		\begin{equation}
			\label{eq:th_pf_nonnegative_diagonal}
			b_{pp\dots p} \geq 0, \quad  \forall \ (p,p,\dots, p) \in \mathscr{Q}_{i_1  \dots i_m}.
		\end{equation}  \label{thmgdd_characterization_item1}
		\item if $(i_1, \dots, i_m) \in \mathscr{F}_n^m$, $\mathcal{B} \in GDD^+_{m,n}$ if and only if $\mathcal{B}$ is a diagonal tensor satisfying $b_{i_1\dots i_m} \geq 0$.
		\label{thmgdd_characterization_item2}
	\end{enumerate}
\end{theorem}
\begin{proof} Let $(i_1,  \dots, i_m) \in \mathscr{D}_{n}^m $ be given. Denote $((j_1, \dots, j_l), \alpha = (\alpha_1,  \dots. \alpha_l))$ as the tight pair associated with $(i_1, \dots, i_m)$. Let $\mathcal{B} \in \mathbb{D}^{i_1 \dots i_m}_{m,n} $. Then, all the off-diagonal elements of $\mathcal{B}$ are zero except for the elements $b_{p_1 \dots p_m}$, where $ (p_1, \dots ,p_m) \in \mathscr{P}_{i_1\dots i_m} $. Then, using Proposition \ref{prop: alter_gdd_def}, it follows that $\mathcal{B} \in GDD_{m,n}^+$  if and only if its entries satisfy~\eqref{eq:th_pf_nonnegative_diagonal} and 
	\begin{equation}
		\label{eq:geo_mean_cone}
		b_{j_k j_k \dots j_k} d_{j_k}^m \geq \binom{m-1}{\alpha - e_k} | b_{i_1 \dots i_m} | d_{i_1} d_{i_2} \dots d_{i_m},
	\end{equation}
	for $k \in [l]$ and some $d_{j_k}>0$, for all $ k \in [l]$, after using \eqref{eq:prop1_prof3}, the sparsity pattern and symmetry of $\mathcal{B}$, and the fact that the number of equal summands in the right-hand side of \eqref{eq:prop1_prof3} in this case is $ \binom{m-1}{\alpha - e_k}$.
	
	Now note that if \eqref{eq:th_pf_nonnegative_diagonal} and \eqref{eq:geo_mean_cone} hold then \eqref{eq:th_pf_nonnegative_diagonal} and
	\begin{equation}
		\label{eq:geo_mean_cone1}
		b_{j_k j_k \dots j_k} ^{\alpha_k} d_{j_k}^{m \alpha_k } \geq \binom{m-1}{\alpha - e_k}^{\alpha_k}  |b_{i_1 \dots i_m}|^{\alpha_k} d_{i_1}^{\alpha_k} d_{i_2}^{\alpha_k} \dots d_{i_m}^{\alpha_k},
	\end{equation}
	hold for all $k \in [l]$, and some $d_{j_k} > 0$, for all $k \in [l]$; since~\eqref{eq:geo_mean_cone1} is obtained by taking the $\alpha_k$th power on both sides of \eqref{eq:geo_mean_cone}, whose (multiplicative) terms are all nonnegative. Given that both the left-hand side and the right-hand side of~\eqref{eq:geo_mean_cone1} are nonnegative, it follows, after multiplying the left-hand sides and the right-hand sides of~\eqref{eq:geo_mean_cone1} for all $k \in [l]$, and using the fact that $\|\alpha\|_1 = m$,  that \eqref{eq:th_pf_nonnegative_diagonal} and~\eqref{eq:geo_mean_cone1} imply \eqref{eq:th_pf_nonnegative_diagonal} and
	\begin{equation} 
		\label{eq:sdd_characterization_pf_1}
		\prod_{k=1}^l (	b_{j_k j_k \dots j_k} ^{\alpha_k} d_{j_k}^{m \alpha_k }) \geq  \left (\prod_{k=1}^l \binom{m-1}{\alpha - e_k}^{\alpha_k} \right ) |b_{i_1  \dots i_m}|^m  (d_{i_1} d_{i_2} \dots d_{i_m})^m,
	\end{equation}
	for some $d_{j_k}>0$, for all $k \in [l]$. In turn, \eqref{eq:sdd_characterization_pf_1} is equivalent to~\eqref{eq:geo_mean_cone_}, with $c :=  \prod_{k=1}^l \binom{m-1}{\alpha - e_k}^{\alpha_k}$, after noticing that from the definition of tight pair~\eqref{eq:tightpair}, it follows that 
	\begin{equation}
		\label{eq:tightpair2}
		\prod_{k=1}^l d_{j_k}^{\alpha_k } = d_{i_1} d_{i_2} \dots d_{i_m}.
	\end{equation}
	Now, to complete the proof, we show that~\eqref{eq:geo_mean_cone_} and~\eqref{eq:th_pf_nonnegative_diagonal} imply~\eqref{eq:geo_mean_cone} (i.e., that $\mathcal{B}$ is a $GDD^+_{m,n}$ tensor). First note that if for any $k \in [l]$, $b_{j_k j_k \dots j_k} = 0$, then~\eqref{eq:geo_mean_cone_} implies that $b_{i_1  \dots i_m}=0$. Thus, in this case, given \eqref{eq:th_pf_nonnegative_diagonal} and the fact that $d_{j_k} > 0$ for all $k \in [l]$, it follows that~\eqref{eq:geo_mean_cone} is satisfied for all $k \in [l]$. Moreover, in the case where $b_{i_1 \dots i_m}=0$, condition~\eqref{eq:geo_mean_cone} follows from~\eqref{eq:th_pf_nonnegative_diagonal}, given the fact that $d_{j_k} > 0$ for all $k \in [l]$. Thus, it is enough to consider the case in which $b_{j_k j_k \dots j_k} > 0$ for all $k \in [l]$, and $b_{i_1  \dots i_m} \not = 0$. In this case, using the fact that $d_{j_k} > 0$, we can write that 
	\begin{equation} 
		\label{eq:ddef}
		d_{j_k}= z \sqrt[m]{\frac{\binom{m-1}{\alpha - e_k}}{b_{j_k j_k \dots j_k} }},
	\end{equation}
	for some $z>0$, for all $k \in [l]$. Thus, for any $k \in [l]$, it follows that 
	\begin{equation} 
		\label{eq:sparse_tensor_pf}
		|b_{i_1 \dots i_m} | d_{i_1}   \dots d_{i_m} = z^m|b_{i_1  \dots i_m}|  \sqrt[m]{ \frac{c}{  \Pi_{k=1}^l b_{j_k j_k\dots j_k}^{\alpha_k} }} \leq z^m = \frac{b_{j_k j_k \dots j_k} d_{j_k}^{m }  }{  \binom{m-1}{\alpha - e_k}},
	\end{equation}
	where the first equality follows by using~\eqref{eq:tightpair2},~\eqref{eq:ddef}, and the definition of $c$; the inequality follows from~\eqref{eq:geo_mean_cone_}, and the last equality follows by using~\eqref{eq:ddef} again. After noticing that~\eqref{eq:sparse_tensor_pf} is equivalent to~\eqref{eq:geo_mean_cone}, it then follows that ~\eqref{eq:geo_mean_cone_} and~\eqref{eq:th_pf_nonnegative_diagonal} imply~\eqref{eq:geo_mean_cone}; that is, that $\mathcal{B} \in GDD^+_{n,m}$.
	
	If $(i_1, \dots, i_m) \in \mathscr{F}_n^m$ and tensor $\mathcal{B} = (b_{p_1  \dots p_m}) \in \mathbb{D}^{i_1   \dots i_m}_{m,n}  $, it follows from the definition of $\mathbb{D}^{i_1   \dots i_m}_{m,n}  $ (i.e.,  \eqref{eq:def_D_mn_i}) that $\mathcal{B}$ is a diagonal tensor in which the only nonzero entry is  $b_{i_1\dots i_m}$. Thus, $\mathcal{B} \in GDD^+_{m,n}$ tensor if and only if $\mathcal{B}$ is a diagonal tensor satisfying $b_{i_1\dots i_m} \geq 0$. 
\end{proof}

Next, in Corollary \ref{col:fullGDDplus}, we apply Theorem~\ref{theo:gdd_decomp} and Theorem~\ref{theo:gdd_characterization} to obtain sufficient and necessary conditions for a tensor $\mathcal{A} \in \mathbb{S}_{m,n}$ to be an $H^+$-tensor (or equivalently a GDD$^+$ tensor). Efforts to characterize $H$-tensors~\citep[see, e.g.,][]{huang2019iterative,li2014criterions,li2017programmable,liu2017iterative,wang2017new, zhang2016h,zhao2016criterions,sun2020new,huang2019some,liu2020iterative} have focused on establishing sufficient conditions for a tensor to qualify as such. While these existing algorithms can detect many $H^+$-tensors, some evade detection. Notably, \citep{luan2019new} employs spectral theory to derive a necessary and sufficient condition for strong $H$-tensors, offering an iterative method with linear convergence.  In contrast, building upon Corollary \ref{col:fullGDDplus}, our approach allows us to take advantage of
interior point methods for power cone optimization, ensuring polynomial time complexity and at least linear convergence~\citep{chares2009cones}. These conditions, derived from diagonal dominance properties, not only aid in identifying symmetric $H^+$-tensors but also enable direct optimization within this tensor class, highlighting the strengths of our method.

\begin{corollary}
	\label{col:fullGDDplus}
	Let $m, n \in \mathbb{N}$. Then $\mathcal{A} = (a_{p_1 p_2 \dots p_m}) \in \mathbb{S}_{m,n}$ is a GDD$^+$ tensor if and only if there exist $b^{\vec{i}}_{j} \ge 0$ for all $\vec{i} = (i_1,\dots,i_m) \in \mathscr{D}_n^m$,
	$j \in \ii$ satisfying
	\begin{enumerate}[label = (\roman*)]
		\item For $\vec{i} \in \mathscr{D}_n^m$, 
		\begin{equation}
			\label{eq:gdd_corollary2}
			\prod_{k=1}^{l^{\ii}} 
			(b^{{\vec{i}}}_{j_k})^{\alpha_k^{\ii}}  \geq c(\ii) | a_{\vec{i}} |^m 
		\end{equation}
		where  $c(\ii) = \prod_{k=1}^{l^{\vec{i}}} {\binom{m-1}{\alpha^{\vec{i}} - e_k} }^{\alpha^{\vec{i}}_k}$, and $((j^{\vec{i}}_1,j^{\vec{i}}_2, \dots, j^{\vec{i}}_{l^{\vec{i}}}),\alpha^{\vec{i}} =  (\alpha^{\vec{i}}_1, \alpha^{\vec{i}}_2, \dots. \alpha^{\vec{i}}_{l^{\vec{i}}}))$ is the tight pair associated with $\vec{i} $.
		\label{col:fullGDDplusitem1}
		\item For $j \in [n]$, 
		\begin{equation}
			\label{eq:gdd_corollary1}
			a_{jj\dots j} \geq  \sum_{\vec{i} \in  \mathscr{D}_{n}^m: j \in \vec{i}} b^{\vec{i}}_{j}.
		\end{equation}
	\end{enumerate}
\end{corollary}
\begin{proof}
	Let $m, n \in \mathbb{N}$. From Theorem~\ref{theo:gdd_decomp}, $\mathcal{A} = (a_{p_1 p_2 \dots p_m}) \in \mathbb{S}_{m,n}$ is a GDD$^+$ tensor if and only if 
	\begin{equation}
		\label{eq:geo_mean_cone_corollary_sum}
		\mathcal{A} = \sum_{\vec{i} \in \mathscr{D}_{n}^m\cup \mathscr{F}_{n}^m}\mathcal{B}^{{\vec{i}}}
	\end{equation} and for $\vec{i} \in \mathscr{D}_{n}^m\cup \mathscr{F}_{n}^m$, $\mathcal{B}^{{\vec{i}}} = (b^{{\vec{i}}}_{p_1 p_2 \dots p_m}) \in \mathbb{D}_{m,n}  \cap GDD_{m,n}^+$ satisfies conditions~\ref{thmgdd_characterization_item1} and~\ref{thmgdd_characterization_item2} in Theorem~\ref{theo:gdd_characterization}. Note that from the sparse structure of the tensors $\mathcal{B}^{\vec{i}}$ used in~\eqref{eq:geo_mean_cone_corollary_sum}, it follows that 
	for any $j \in [n]$,	
	\begin{equation}
		\label{eq:diagelements}
		a_{jj\dots j} = \sum_{\vec{i} \in \mathscr{D}_{n}^m: (j,j,\dots, j) \in \mathscr{Q}_{\vec{i}}} b^{\vec{i}}_{jj\dots j} + 
		b^{jj\dots j}_{jj\dots j},
	\end{equation}
	and 
	for any $\vec{i} \in \mathscr{D}_{n}^m$,
	\begin{equation}
		\label{eq:offdiagelements}
		a_{\vec{i}} = b^{\vec{i}}_{\vec{i}}.
	\end{equation}
	From Theorem~\ref{theo:gdd_characterization}\ref{thmgdd_characterization_item1} and~\eqref{eq:offdiagelements}, it follows that \[c(\ii)|a_{\vec{i}}|^m =c(\ii)|b^{\vec{i}}_{\vec{i}}|^m\leq \prod_{k=1}^{l^{\ii}} 
	(b^{{\vec{i}}}_{j_kj_k\dots j_k})^{\alpha_k^{\ii}}\]
	where  $c(\ii) = \prod_{k=1}^{l^{\vec{i}}} {\binom{m-1}{\alpha^{\vec{i}} - e_k} }^{\alpha^{\vec{i}}_k}$,  $((j^{\vec{i}}_1,j^{\vec{i}}_2, \dots, j^{\vec{i}}_{l^{\vec{i}}}),\alpha^{\vec{i}} =  (\alpha^{\vec{i}}_1, \alpha^{\vec{i}}_2, \dots. \alpha^{\vec{i}}_{l^{\vec{i}}}))$ is the tight pair associated with $\vec{i}$, and $b^{\ii}_{pp\dots p} \geq 0$, for all $(p,p,\dots, p) \in \mathscr{Q}_{\ii}$. The statement then follows from this and~\eqref{eq:diagelements}, after noticing that from 
	Theorem~\ref{theo:gdd_characterization}\ref{thmgdd_characterization_item2}, $b^{jj\dots j}_{jj\dots j} \ge 0$ for all $j \in [n]$, and after simplifying notation to let $b^{\vec{i}}_{jj\dots j} := b^{\vec{i}}_j$ for any $\vec{i} \in  \mathscr{D}_{n}^m: (jj\dots j) \in \mathscr{Q}_{\vec{i}}$; that is, for any $\vec{i} \in  \mathscr{D}_{n}^m: j \in \vec{i}$.
\end{proof}

Now we provide an example to illustrate the results in Theorem~\ref{theo:gdd_decomp} and 
Corollary~\ref{col:fullGDDplus}.
\begin{example}
	Consider the following symmetric tensor
	\[
	\mathcal{A} = (a_{i_i i_2 i_3 i_4}) = [ A(1,1,:,:), A(1,2,:,:); A(2,1,:,:), A(2,2,:,:) ] \in \mathbb{S}_{4,2},
	\]
	where 
	\begin{equation*}
		A(1,1,:,:) = \begin{pmatrix}
			4& -2 \\
			-2 & -1
		\end{pmatrix}, 
		A(1,2,:,:) =\begin{pmatrix}
			-2 & - 1\\
			-1 & 64/3
		\end{pmatrix},
	\end{equation*}
	\begin{equation*}
		A(2,1,:,:) = \begin{pmatrix}
			-2 & -1 \\
			-1 & 64/3
		\end{pmatrix},
		A(2,2,:,:) = \begin{pmatrix}
			-1 & 64/3 \\
			64/3 &1000
		\end{pmatrix}.
	\end{equation*}
	
	Denote $D_1 = \begin{pmatrix}
		1 & 0 \\
		0 & 2
	\end{pmatrix} $, $D_2 = \begin{pmatrix}
		1/2 & 0 \\
		0 & 2
	\end{pmatrix} $, $D_3 = \begin{pmatrix}
		1/3 & 0 \\
		0 & 4
	\end{pmatrix} $. Then, one can obtain \[
	\mathcal{A} = \frac{1037}{1296}\mathcal{V}^{0,1111} + 168\mathcal{V}^{0,2222} +  \mathcal{B}^{(1112)} +  \mathcal{B}^{(1122)}+  \mathcal{B}^{(1222)},
	\]
	where
	\[
	\mathcal{B}^{(1112)} = (b^{(1112)}_{j_1 j_2 j_3 j_4}) =  \mathcal{V}^{1,1112}D_1 D_1 D_1 D_1,
	\]
	\[
	\mathcal{B}^{(1122)}= (b^{(1122)}_{j_1 j_2 j_3 j_4})= \mathcal{V}^{1,1122}D_2 D_2 D_2 D_2,
	\]
	\[
	\mathcal{B}^{(1222)} = (b^{(1222)}_{j_1 j_2 j_3 j_4})= \mathcal{V}^{0,1222}D_3D_3 D_3 D_3.
	\]
	Let $b^{\vec{i}}_{j}$ = $b^{\vec{i}}_{jjjj}  \geq 0, j \in  {\vec{i}}$ for $\vec{i} \in \mathscr{D}_2^4$. Then it is easy to show that these $b^{\vec{i}}_{j}, j \in  {\vec{i}}, \vec{i} \in \mathscr{D}_2^4$ satisfy~\eqref{eq:gdd_corollary1} and~\eqref{eq:gdd_corollary2}. As a result, from Corollary \ref{col:fullGDDplus}, $\mathcal{A}$ is a symmetric $H^+$-tensor (GDD$^+$ tensor). In Section~\ref{sec:powercone}, we will show that Theorem \ref{th:check_gdd} allows us to obtain the matrices $D_1$, $D_2$ and $D_3$ by solving a power cone optimization problem~\citep{chares2009cones}.

	On the other hand, denote $D = \begin{pmatrix}
		3  & 0 \\
		0 & 1/2
	\end{pmatrix} $. Then \[
	\bar{\mathcal{A}} = \mathcal{A}DDDD =  [ \bar{A}(1,1,:,:), \bar{A}(1,2,:,:); \bar{A}(2,1,:,:), \bar{A}(2,2,:,:) ] ,
	\]
	where \begin{equation*}
		\bar{A}(1,1,:,:) = \begin{pmatrix}
			324 & -27 \\
			-27 & -9/4
		\end{pmatrix}, 
		\bar{A}(1,2,:,:) = \begin{pmatrix}
			-27 & -9/4 \\
			-9/4 & 8
		\end{pmatrix},
	\end{equation*}
	\begin{equation*}
		\bar{A}(2,1,:,:) = \begin{pmatrix}
			-27 & -9/4 \\
			-9/4 & 8
		\end{pmatrix},
		\bar{A}(2,2,:,:) = \begin{pmatrix}
			-9/4 & 8 \\
			8 &125/2
		\end{pmatrix},
	\end{equation*}
	is a DD$^+$ tensor. Thus, from Definition~\ref{def:dd_sdd}\ref{defitemii}, $\mathcal{A}$ is a symmetric $H^+$-tensor (GDD$^+$ tensor).
\end{example}

\subsection{Identifying symmetric $H^+$-tensors with power cone optimization}
\label{sec:powercone}

Corollary~\ref{col:fullGDDplus} readily implies that one can identify whether a 
symmetric tensor is an 
$H^+$-tensors  using tractable conic optimization techniques, and more precisely, the {\em power cone}~\citep[see, e.g.,][]{chares2009cones,hien2015differential}. To illustrate this, let us first introduce the {\em high-dimensional power cone}.

\begin{definition}[{High-dimensional power cone~\citep[][Sec. 4.1.2]{chares2009cones}}]
	\label{def:highpowercone}
	For any $\alpha \in \R^m_+$ such that $e\tr \alpha = 1$, the {\em  high-dimensional power cone} is defined by
	\begin{equation}
		\label{eq:highpowercone}
		\mathbb{K}^{(m)}_{\alpha} = \{(x, z) \in \R^m_+ \times \R: x_1^{\alpha_1} \cdots x_m^{\alpha_m} \ge |z| \}.
	\end{equation}
\end{definition}

Now, for any tensor $\mathcal{A} \in \mathbb{S}_{m,n}$, let

\begin{flalign}
	\label{eq:F_A}
	\mathbb{F}(\mathcal{A}) = \left \{ d^{\vec{i}}_{j} \in \R, \vec{i} \in \mathscr{D}_n^m, j \in \ii :
	\begin{array}{ll}
		a_{jj\dots j} \geq  \dsum_{\vec{i} \in  \mathcal{D}_{n}^m: j \in \ii} d^{\vec{i}}_{j}, &\forall j \in [n]\\
		( d^{\vec{i}}_{{i}_1}, \dots,  d^{\vec{i}}_{{i}_m}, c(\ii)^{\frac{1}{m}}a_{\ii}) \in \mathbb{K}^{(m)}_{\frac{1}{m}e}, & \forall  \vec{i} \in \mathscr{D}_n^m \\
	\end{array}  \right \}.
\end{flalign}

The next Corollary then follows from Definition~\ref{def:highpowercone} and Corollary~\ref{col:fullGDDplus}.

\begin{corollary}
	\label{cor:F(A)}
	Let $m, n \in \mathbb{N}$. Then $\mathcal{A} = (a_{p_1 p_2 \dots p_m}) \in \mathbb{S}_{m,n}$ is a GDD$^+$ tensor if and only if $\mathbb{F}(\mathcal{A}) \neq \emptyset$.
\end{corollary}

Furthermore, the condition $\mathbb{F}(\mathcal{A}) \neq \emptyset$ in Corollary~\ref{cor:F(A)} can be checked in polynomial time using appropriate {\em interior point methods}~\citep[see, e.g.,][]{renegar2001mathematical}. To show this, we make use of the power cone, which is a lower-dimensional version of the high-dimensional power cone introduced in Definition~\ref{def:highpowercone}. Namely, 
for any $\alpha \in [0,1]$, the power cone $\mathbb{K}_{\alpha} := \mathbb{K}^2_{\alpha, 1-\alpha} = \{ (x,z) \in \R^2_+ \times \R : x_1^{\alpha} x_2^{1-\alpha} \ge |z| \}$~\citep[see, e.g.][]{koecher1957positivitatsbereiche, nesterov2012towards,roy2022self}. 
As shown in~\citep[][eq. (4.3), Sec. 4.1.2]{chares2009cones}, the higher-dimensional power cone $\mathbb{K}^{(m)}_{\alpha}$ can be decomposed into $m-1$ (low-dimensional) power cones. 

Using this fact, we can rewrite~\eqref{eq:F_A} as follows:
\begin{flalign}
	\label{eq:F_A_long}
	\mathbb{F}(\mathcal{A}) = &\left \{ 
	\begin{array}{l}
		d^{\vec{i}}_{j} \in \R, \vec{i} \in \mathscr{D}_n^m, j \in \ii\\
		v^{\vec{i}}_l \in \R_+, \vec{i} \in \mathscr{D}_n^m, l \in [m-2]\\
	\end{array}: \nonumber
	\right . \\ & \hspace{2cm} { \left . 
	\begin{array}{ll}
		a_{jj\dots j} \geq  \dsum_{\vec{i} \in  \mathscr{D}_{n}^m: j \in \ii} d^{\vec{i}}_{j}, &\forall j \in [n],\\
		( d^{\vec{i}}_{{i}_1}, v^{\ii}_1, c(\ii)^{\frac{1}{m}}a_{\ii}) \in \mathbb{K}_{\frac{1}{m}}, & \forall  \vec{i} \in \mathscr{D}_n^m \\
		( d^{\vec{i}}_{{i}_l}, v^{\ii}_l, v^{\ii}_{l-1}) \in \mathbb{K}_{\frac{1}{m-l+1}}, & \forall  \vec{i} \in \mathscr{D}_n^m, l=2, \dots, m-2 \\
		( d^{\vec{i}}_{{i}_{m-1}},  d^{\vec{i}}_{{i}_m}, v^{\ii}_{m-2}) \in {\mathbb{K}_{\frac{1}{2}}}, & \forall  \vec{i} \in \mathscr{D}_n^m \\
	\end{array}  \right \} }.
\end{flalign}

The relevance of introducing the power cone in~\eqref{eq:F_A_long} is that~\citep{chares2009cones,nesterov2012towards,roy2022self} provide  {\em self-concordant barriers} for the power cone. In short, this means that for any $\mathcal{A} \in \mathbb{S}_{m,n}$, the nonsymmetric conic feasibility system defined by~\eqref{eq:F_A_long}  can be solved in polynomial time using a {\em primal-dual predictor-corrector method}~\citep{wright1997primal}. The reference to nonsymmetry, stems from the fact that the power cone is not symmetric if $\alpha \neq \frac{1}{2}$~\citep{hien2015differential,tunccel2010self}. Open source software such as \texttt{SCS}~\citep{ocpb:16}, \texttt{Hypatia}~\citep{coey2022solving}, \texttt{DDS}~\citep{karimi2024domain}, \texttt{alfonso}~\citep{papp2022alfonso}, \texttt{Clarabel}~\citep{Clarabel_2024} and the commercial solver \texttt{MOSEK}~\citep{mosek2022} are powerful tools for solving power cone optimization problems. In particular, current solvers can handle power cones of size 2,500 in just 24 milliseconds~\citep[][Table~1]{chen2023efficient}. A GPU solver implementation, {\tt CuClarabel}, for power cone optimization problems has also recently emerged~\citep{chen2024cuclarabel}. Furthermore, the high-dimensional power cone can be represented by an {\em exponential cone} \citep{henrik2024}, which allows to leverage solvers for exponential cone optimization problems. This work leverages \texttt{SPOT}~\citep{SPOT} for formulating the power cone optimization problems, and \texttt{MOSEK 9.3.22}~\citep{mosek2022} is used to efficiently solve the resulting optimization tasks. 

\begin{theorem}
	\label{th:check_gdd}
	For $m, n \in \mathbb{N}$, to check if a tensor in $\mathbb{S}_{m,n}$ is an $H^+$-tensor (GDD$^+$ tensor) is equivalent to solve a power cone optimization problem of size polynomial in $n$ for a fixed $m$.
\end{theorem}
\begin{proof}
	The result follows from Corollary~\ref{cor:F(A)}, equation~\eqref{eq:F_A_long}, and the fact that $|\mathscr{D}_{n}^m | = \binom{n+m-1}{m} - n$.
\end{proof}

For a detailed discussion of the properties of, and optimization over the power cone, we direct the reader to~\citep{mosek,chares2009cones}.

As mentioned earlier, even order symmetric $H^+$-tensors are PSD tensors~\citep{chen2015sos}. As we will demonstrate below, this property enables the introduction of a novel class of nonnegative polynomials, which can be used to address the solution of polynomial optimization (PO) problems; that is, problems whose objective and constraints can be defined by polynomials. PO is an area that takes advantage of algebraic geometric results to construct hierarchies of convex optimization problems that provide increasingly tight approximations of the PO problem. The most common approach is to draw on properties of sums of squares (SOS) polynomials to construct the desired hierarchies using semidefinite optimization (SDO), as checking if a polynomial is SOS is equivalent to solving a SDO~\citep[see, e.g.,][]{lasserre2015introduction}. However, solving the associated SDO problems is in general prohibitively expensive in terms of computational effort. As a result, a direction of research in PO now focuses on using new classes of nonnegative polynomials that might lead to hierarchies that are constructed using other optimization techniques such as linear optimization or second-order cone optimization~\citep[see, e.g.,][]{kuryatnikova2024reducing, ahmadi2019dsos}. As shown below, even order symmetric $H^+$-tensors (GDD$^+$ tensor) provide a way to create such a class of nonnegative polynomials. For that purpose, we begin by defining the set of polynomials derived from symmetric $H^+$-tensors (GDD$^+$ tensor).

\begin{definition}
	\label{def:sddtpsd}
	A polynomial $p(x)  \in \R[x]$ with degree $m$ and $n$ variables is called GDDTSOS if there is a tensor $\mathcal{A} = (a_{p_1 p_2 \dots p_m}) \in GDD^+_{m,n+1}$ such that 
	$p(x) = \langle A, x\otimes \dots \otimes x \rangle$, where $x = (1, x_1,x_2,\dots, x_n)^T$.
\end{definition}


\begin{example}[Application in polynomial optimization]
	\label{example:pop}
	For $m, n \in \mathbb{N}$, let 
	\[
	\mathcal{K}_{2m,n} = \{ p(x) \in \mathbb{R}[x]: p(x) \text{ is GDDTSOS with degree $2m$ and $n$ variables} \}.
	\]
	Then  $\mathcal{K}_{2m,n} \supset \mathbb{R}_+ $ and $\mathcal{K}_{2m,n}$ is contained in the set of nonnegative polynomials. To see that $\mathcal{K}_{2m,n} \supset \mathbb{R}_+ $ , notice that for $m, n \in \mathbb{N}$ and any $c \in \mathbb{R}_+ $, if we let $\mathcal{A} = (a_{i_1 i_2, \cdots, i_{2m}}) \in  \mathbb{S}_{2m,n+1}$ be the tensor with $a_{11,\cdots,1} = c$ and all the other entries are 0. Then $c = \langle \mathcal{A}, x\otimes \dots \otimes x \rangle$, where $x = (1, x_1,x_2,\dots, x_n)^T$. Clearly,  $\mathcal{A}  \in GDD^+_{2m,n+1} $ and thus $\mathcal{K}_{2m,n} \supset \mathbb{R}_+$. From~\citep{chen2015sos}, even order symmetric $H^+$-tensors are PSD tensors. From~\citep{kannan2015some}, $GDD_{2m, n+1}^+$ is equivalent to the set of symmetric $H^+$-tensors with order $2m$ and dimension $n+1$. Thus, tensors in $GDD_{2m, n+1}^+$ are also PSD tensors and $\mathcal{K}_{2m,n}$ is contained in the set of nonnegative polynomials for $m, n \in \mathbb{N}$.
From Proposition~3.5 and Remark~2 in~\citep{kuryatnikova2024reducing}, it follows that $\mathcal{K}$ satisfies the properties required to construct hierarchies of convex optimization problems that can be solved using power cone optimization, rather than semidefinite optimization,  to approximate any polynomial optimization problem with compact feasible set. This type of approximation approach can be used to address problems in statistics and machine learning, derivative pricing, and control theory~\citep{ahmadi2019dsos}.
\end{example}

%
%


\section{Minimum $H$-eigenvalue of $M$-tensors}
\label{sec:Meigenvalue}

The problem of obtaining bounds on the minimum $H$-eigenvalue of $M$-matrices and $M$-tensors has received significant attention in the literature~\citep{he2014inequalities,huang2018some,li2013new, tian2010inequalities}. This is due to the important role the $M$-tensors play in a wide range of interesting applications~\citep[see,][and the references therein]{huang2018some}. For example, $M$-tensors are used to encode systems of multilinear equations arising in the  numerical solution of partial differential equations, as well as data mining and tensor complementarity problems~\citep{han2017homotopy}.
However, these bounds are loose~\citep[see, e.g.,][Table~1]{huang2018some}, and 
even expensive to compute~\citep[see, e.g.,][Table~2]{huang2018some}. 
Further, the minimum $H$-eigenvalue of $M$-matrices can be computed with {\em homotopy continuation} type algorithms that allow the more general computation of complex generalized tensor eigenpairs~\citep{chen2016computing}. However, these algorithms are not guaranteed to work in polynomial time. 
Next, we show that the characterization in Corollary~\ref{cor:F(A)} can be applied to obtain the exact minimum $H$-eigenvalue of symmetric $M$-tensors in polynomial time by solving a power cone optimization problem. Besides, this result can also be used to obtain lower bounds for the minimum $H$-eigenvalue of {\em general} (i.e., not necessarily symmetric) $M$-tensors in polynomial time by solving a power cone optimization problem. For that purpose, we first introduce the following results.

\begin{lemma}[{\citep[][Lem.~2.2]{zhang2014m}}] 
	\label{lemma:linear_eigenvalue_tensor}
	For $m,n \in \mathbb{N}$, let $\mathcal{A} \in \mathbb{T}_{m,n}$. Suppose that $\mathcal{B} = a(\mathcal{A} + b \mathcal{I})$, where $a$ and $b$ are two real numbers. Then $\mu$ is an eigenvalue ($H$-eigenvalue) of $\mathcal{B}$ if and only if $\mu = a (\lambda +b)$ and $\lambda$ is an eigenvalue ($H$-eigenvalue) of $\mathcal{A}$.
\end{lemma}

The following two results, Lemma~\ref{lemma:H-eigenvalue_Z-tensor} and Proposition~\ref{prop:M-tensor_eigenvalue}, can be derived from~{\citep[][Thm.~3.9, Cor.~3.10, Thm.~3.11]{zhang2014m}}. However, we provide brief proofs of these results for clarity.
\begin{lemma}
	\label{lemma:H-eigenvalue_Z-tensor}
	For $m,n \in \mathbb{N}$, if $\mathcal{A} = s\mathcal{I} - \mathcal{D}\in \mathbb{S}_{m,n}$ where $\mathcal{D}$ is a nonnegative tensor and $s$ is a scalar, then $s - \rho(\mathcal{D})$ is the minimum $H$-eigenvalue of $\mathcal{A}$.
\end{lemma}
\begin{proof}
	First, from Theorem \ref{lemma:nonnegative_tensor_pho} it follows that $\rho(\mathcal{D})$ is an $H$-eigenvalue of $\mathcal{D}$. Then, from Lemma \ref{lemma:linear_eigenvalue_tensor}, $s - \rho(\mathcal{D})$ is an $H$-eigenvalue of $\mathcal{A}$. Assume that $\lambda$ is an $H$-eigenvalue of $\mathcal{A}$. Then, $s - \lambda$ is an $H$-eigenvalue of $\mathcal{D}$. Thus, $\rho(\mathcal{D}) \geq |s - \lambda| \geq s - \lambda $. That is, $\lambda \geq s - \rho(\mathcal{D})$. Thus, $s - \rho(\mathcal{D})$ is the minimum $H$-eigenvalue of $\mathcal{A}$.
\end{proof}

In what follows, for any $\mathcal{A}  \in \mathbb{S}_{m,n}$, let $\lambda_{\min}(\mathcal{A})$ denote the minimum $H$-eigenvalue of $\mathcal{A}$.

\begin{proposition}
	\label{prop:M-tensor_eigenvalue}
	For $m,n \in \mathbb{N}$, if $\mathcal{A} \in \mathbb{S}_{m,n}$ is a $Z$-tensor, then for any $\lambda \leq \lambda_{\min}(\mathcal{A})$, $\mathcal{A} - \lambda I$ is an $M$-tensor. Besides, for any $\lambda > \lambda_{\min}(\mathcal{A})$, $\mathcal{A} - \lambda I$ is not an $M$-tensor.
\end{proposition}
\begin{proof}
	Since $\mathcal{A} \in \mathbb{S}_{m,n}$ is a $Z$-tensor, then there exist a nonnegative tensor $\mathcal{D}$ and nonnegative scalar $s$ such that $\mathcal{A} = sI - \mathcal{D}$. Then, for any $\lambda \leq   \lambda_{\min}(\mathcal{A})$,
	\[
	\mathcal{A} - \lambda I = (s - \lambda)I - \mathcal{D}.
	\]
	From Lemma \ref{lemma:H-eigenvalue_Z-tensor}, $\lambda_{\min}(\mathcal{A}) = s - \rho(\mathcal{D})$. Thus for any $\lambda \leq   \lambda_{\min}(\mathcal{A})$, $s - \lambda -  \rho(\mathcal{D}) \geq s - \lambda_{\min}(\mathcal{A}) -  \rho(\mathcal{D}) =  0$. Furthermore, $s - \lambda \geq \rho(\mathcal{D})  \geq 0$. As a result, $\mathcal{A} - \lambda I$ is an $M$-tensor. Now, for some $\lambda > \lambda_{\min}(\mathcal{A})  $, assume $\mathcal{A} - \lambda I$ is an $M$-tensor. Then there exist a nonnegative tensor $\tilde{\mathcal{D}}$ and nonnegative scalar $\tilde{s} \geq \rho(\tilde{\mathcal{D}})$ such that $\mathcal{A} - \lambda I = \tilde{s} I -\tilde{\mathcal{D}}$. Thus $\mathcal{A} = (\lambda + \tilde{s}) I  - \tilde{\mathcal{D}}$. From Lemma \ref{lemma:H-eigenvalue_Z-tensor}, $\lambda_{\min}(\mathcal{A}) = (\lambda + \tilde{s}) - \rho(\tilde{\mathcal{D}})  \geq \lambda$ which contradicts the condition $\lambda > \lambda_{\min}(\mathcal{A})$. Thus, $\mathcal{A} - \lambda I$ is not an $M$-tensor.
\end{proof}

Note that from Corollary~\ref{cor:F(A)}  and the definition of $H^+$-tensors in terms of the comparison tensor (cf., \eqref{eq:comptensor}), one obtains the following characterization for symmetric $M$-tensors.

\begin{corollary}
	\label{cor:F_M(A)}
	Let $m, n \in \mathbb{N}$. Then $\mathcal{A} = (a_{i_1 i_2 \dots i_m}) \in \mathbb{S}_{m,n}$ is an $M$-tensor if and only if 
	$a_{i_1 i_2 \dots i_m}\le 0$ for all $(i_1,i_2, \dots, i_m) \in  \mathscr{D}_n^m$, and $\mathbb{F}(\mathcal{A}) \neq \emptyset$.
\end{corollary}

Proposition~\ref{prop:M-tensor_eigenvalue}, the characterization of symmetric $M$-tensors in Corollary~\ref{cor:F_M(A)}, and~\eqref{eq:F_A_long}, readily provide a way to compute the minimum $H$-eigenvalue of symmetric $Z$-tensors in polynomial time by solving a power cone optimization problem.

\begin{corollary}
	\label{cor:Meigen}
	For $m,n \in \mathbb{N}$, if $\mathcal{A} \in \mathbb{S}_{m,n}$ is a $Z$-tensor, then 
	\begin{equation}
		\label{eq:M-tensor_eigenvalue_compute}
		\lambda_{\min}(\mathcal{A}) =  {\max} \left\{  \lambda:  \mathbb{F}(\mathcal{A} - \lambda \mathcal{I}) \neq \emptyset \right\}.
	\end{equation}
\end{corollary}
\begin{proof}
	From Proposition~\ref{prop:M-tensor_eigenvalue}, it follows that
	\[
	\lambda_{\min}(\mathcal{A}) =  {\max} \left\{  \lambda:  \mathcal{A} - \lambda \mathcal{I} \text{ is an $M$-tensor} \right\}.
	\]
	Then using Corollary~\ref{cor:F_M(A)} to characterize the set of symmetric $M$-tensors, it follows that 
	
	\begin{equation}
		\label{eq:M-tensor_eigenvalue_compute_pf}
		\lambda_{\min}(\mathcal{A}) = \max \left\{
		\begin{aligned}
			& \lambda : \mathbb{F}(\mathcal{A} - \lambda \mathcal{I}) \neq \emptyset, \\
			& (\mathcal{A} - \lambda \mathcal{I})_{i_1 i_2 \dots i_m} \le 0, \forall ({i_1, i_2, \dots, i_m}) \in \mathscr{D}_n^m
		\end{aligned}
		\right\}.
	\end{equation}
	If $\mathcal{A}$ is a symmetric $Z$-tensor, then for any $\lambda \in \mathbb{R}$,
	\[
	(\mathcal{A} - \lambda \mathcal{I})_{i_1 i_2 \dots i_m} \le 0, \forall ({i_1,i_2, \dots, i_m}) \in  \mathscr{D}_n^m.
	\]
	Thus, one can simplify \eqref{eq:M-tensor_eigenvalue_compute_pf} and obtain \eqref{eq:M-tensor_eigenvalue_compute} for a symmetric $Z$-tensor $\mathcal{A}$.
\end{proof}

 Furthermore, according to~\citep{ding2013m} and Theorem~3.3 in~\citep{zhang2012m}, a $Z$-tensor $\mathcal{A}$ is a strong $M$-tensor if and only if $\lambda_{\min}(\mathcal{A}) > 0$. Consequently, \eqref{eq:M-tensor_eigenvalue_compute} can also be used to determine whether a symmetric $Z$-tensor is a strong $M$-tensor.

Symmetric $M$-tensors are all symmetric $Z$-tensors. Thus, Equation~\eqref{eq:F_A_long}, and the discussion that follows it, mean that one can compute the minimum $H$-eigenvalue of a symmetric $M$-tensor by solving the power cone optimization problem~\eqref{eq:M-tensor_eigenvalue_compute}. To benchmark the performance of the proposed method, we apply it to 
obtain the minimum $H$-eigenvalue of the symmetrized\footnote{Tensor $\mathcal{A}:=\text{sym}(\mathcal{B})$ is called the symmetrized version of tensor $\mathcal{B}$ if their corresponding polynomials are the same and $\mathcal{A}$ is a symmetric tensor. In what follows, for $\mathcal{A} \in \mathbb{T}_{m,n}$, denote sym($\mathcal{A} $) as the symmetrized version of~$\mathcal{A} $.} $M$-tensors considered in 
Example~3.1 and Example~3.2 in~\citep{huang2018some}. Specifically, in Table~\ref{tab:bounds}, we compare the best upper and lower bounds for the minimum $H$-eigenvalue of the symmetrized $M$-tensors obtained using the methodologies proposed in~\citep{huang2018some}, versus the value of the minimum $H$-eigenvalue of these $M$-tensors obtained using~\eqref{eq:M-tensor_eigenvalue_compute}.

\begin{table}[!htb]
	\begin{center}
		\begin{tabular}{
				>{\centering\arraybackslash}m{2.0cm}
				>{\centering\arraybackslash}m{0.3cm}
				>{\centering\arraybackslash}m{0.3cm}
				>{\centering\arraybackslash}m{2cm}
				>{\centering\arraybackslash}m{1.5cm}
				>{\centering\arraybackslash}m{2cm}}
			\toprule
			& & & \multicolumn{3}{c}{minimum $H$-eigenvalue} \\
			\cmidrule{4-6}
			\multicolumn{1}{c}{symmetrized} & & & \multicolumn{1}{c}{best lower bound} & & \multicolumn{1}{c}{best upper bound} \\
			\multicolumn{1}{c}{$M$-tensor} & \multicolumn{1}{c}{$m$} & \multicolumn{1}{c}{$n$} & \multicolumn{1}{c}{\citep{huang2018some}} & \multicolumn{1}{c}{value \eqref{eq:M-tensor_eigenvalue_compute}} & \multicolumn{1}{c}{\citep{huang2018some}} \\
			\midrule
			Example~3.1 in~\citep{huang2018some} & 3 & 3 & 1.1196 & 4.4404 & 6.9383 \\
			Example~3.2 in~\citep{huang2018some} & 3 & 3 & 2.6088 & 6.3122 & 9.1984 \\
			\bottomrule
		\end{tabular}
	\end{center}
	\caption{Minimum $H$-eigenvalues of symmetric $M$-tensors. \label{tab:bounds}}
\end{table}

The results in Table~\ref{tab:bounds} show that, neither the lower or upper bounds for the minimum $H$-eigenvalues resulting from the results in~\citep{huang2018some} are particularly tight in comparison with the actual minimum $H$-eigenvalues. 

To show the efficiency of the proposed method on computing the minimum $H$-eigenvalues of symmetric $M$-tensors, we also compare the proposed method with the method in~\citep{chen2016computing}; namely, a {\em homotopy continuation} type algorithm that finds complex generalized eigenpairs. This method combines a heuristic approach and a Newton homotopy method to extract real eigenpairs. The advantage of this method is that it works for general tensors. However, eigenvalue computation is very difficult for third or higher order tensors~\citep{hillar2013most}. In general, the algorithm proposed in~\citep{chen2016computing} is not guaranteed to work in polynomial time. 
Thus, compared to the method proposed in this work, which is particularly designed for symmetric $M$-tensors, the method in~\citep{chen2016computing} is very inefficient when it is applied to this class of tensors. The results in Table \ref{tab:comparison2} show that both the method in~\citep{chen2016computing} and the method proposed here, return the same minimum $H$-eigenvalues for symmetric $M$-tensors. However, the computation time of the method proposed here is about four times lower than the one of the method in~\citep{chen2016computing}. Actually, the time used in solving the corresponding power cone optimization problems of the method proposed here 
is an order of magnitude lower than the time used by the method in~\citep{chen2016computing} (see the numbers in parenthesis of the last column in Table~\ref{tab:comparison2}). That is, most of the computational time used to implement the methodology proposed here is spent constructing the actual power cone optimization problems that need to be solved.
Thus, with a better optimization model formulation framework (currently we use SPOT~\citep{SPOT} which is quite inefficient), the total solution time of the method proposed here can be improved a lot. In conclusion, compared to the method in~\citep{chen2016computing}, the proposed method in this work is both theoretically and empirically more efficient in obtaining the minimum $H$-eigenvalue of symmetric $M$-tensors.


\begin{table}[!htb]
	\begin{center}
		\begin{tabular}{
				>{\centering\arraybackslash}m{1.8cm}
				>{\centering\arraybackslash}m{0.2cm}
				>{\centering\arraybackslash}m{0.2cm}
				>{\centering\arraybackslash}m{1.7cm}
				>{\centering\arraybackslash}m{0.8cm}
				>{\centering\arraybackslash}m{1.7cm}
				>{\centering\arraybackslash}m{1.3cm}}
			\toprule
			& & & \multicolumn{2}{c}{minimum $H$-eigenvalue} & \multicolumn{2}{c}{solution time} \\
			\cmidrule{4-5} \cmidrule{6-7}
			\multicolumn{1}{c}{symmetrized} & & & \multicolumn{1}{c}{value} & \multicolumn{1}{c}{value} & \multicolumn{1}{c}{time (s)} &  \multicolumn{1}{c}{time (s)} \\
			\multicolumn{1}{c}{$M$-tensor} & \multicolumn{1}{c}{$m$} & \multicolumn{1}{c}{$n$} & \multicolumn{1}{c}{ \citep{chen2016computing}} &\multicolumn{1}{c}{\eqref{eq:M-tensor_eigenvalue_compute}} &\multicolumn{1}{c}{\citep{chen2016computing}} &\multicolumn{1}{c}{\eqref{eq:M-tensor_eigenvalue_compute}} \\
			\midrule
			Example~3.1 \newline in~\citep{huang2018some} & 3 & 3 & 4.4404 & 4.4404 & 0.2681 & 0.0623 (0.0169) \\
			Example~3.2 \newline in~\citep{huang2018some} & 3 & 3 & 6.3122 & 6.3122 & 0.2715 & 0.0649 (0.0225) \\
			\bottomrule
		\end{tabular}
	\end{center}
	\caption{Minimum $H$-eigenvalues of symmetric $M$-tensors and solution time using~\eqref{eq:M-tensor_eigenvalue_compute} and the method in~\citep{chen2016computing} (The numbers in parenthesis in the last column are the times used in solving the corresponding power cone optimization problems). \label{tab:comparison2}}
\end{table}

The proposed method to compute the minimum $H$-eigenvalue of symmetric $M$-tensors can also be used to compute the largest $H$-eigenvalue of symmetric nonnegative tensors. Assume $\mathcal{D}$ is a symmetric nonnegative tensor. From Lemma~\ref{lemma:H-eigenvalue_Z-tensor}, $\rho(\mathcal{D})$ can be obtained by computing the minimum $H$-eigenvalue of $-\mathcal{D}$ using~\eqref{eq:M-tensor_eigenvalue_compute}. 
To illustrate the practical applications of this proposed method, we show how it can be used to obtain an upper bound for the chromatic number of a {\em hypergraph}~\citep[see, e.g.,][]{chang2013survey}. Before presenting the application, we introduce some definitions related to hypergraphs. For more details, we refer the reader to~\citep{chang2013survey,qi2013h,cooper2012spectra}.

\begin{definition}[{\citep[][Def. 6.1]{chang2013survey}}]
	\label{def:hypergraph}
	A hypergraph $\mathcal{H}$ is pair of $(V, E)$ where $E \in \mathcal{P}(V)$, the power set of $V$. The elements of $V = V(\mathcal{H})$ are called vertices, and the elements of $E = E(\mathcal{H})$ are called edges. A hypergraph is said to be $k$-uniform for an integer $k \geq 2$, if for any $e \in E(\mathcal{H})$, the cardinality of the subset, $card(e) = k$. 
\end{definition}

\begin{definition}[{\citep[][Def.~6.2]{chang2013survey}}]
	\label{def:adjacency_tensor_hypergraph}
	The adjacency tensor $\mathcal{A}_{\mathcal{H}}$ for a $m$-uniform hypergraph $H = (V, E)$, denoted as $\mathcal{A}_\mathcal{H} = (a_{i_1\cdots i_m}) \in \mathbb{S}_{m,n}$, where $n$ is the number of nodes in set $V$, is the symmetric tensor given by 
	\begin{equation}
		\label{eq:adjacency_tensor}
		\mathcal{A}_{\mathcal{H}} = \frac{1}{(m-1)!}\begin{cases}
			1 & \text{if } \{i_1, \cdots, i_m\} \in E, \\
			0 & \text{otherwise}.
		\end{cases}
	\end{equation}
\end{definition}
For a hypergraph $\mathcal{H}$, a function $f:  V(\mathcal{H}) \rightarrow [r]$ is a (weak) proper $r$-coloring of $\mathcal{H}$ if for every edge $\{v_1, v_2, \cdots, v_k\}$, there exist $i \neq  j $ such that $f(v_i) \neq f(v_j)$. The (weak) {\em  chromatic number} of $\mathcal{H}$, denoted $\chi(\mathcal{H})$, is the minimum $r$ such that $\mathcal{H}$ has a proper $r$-coloring. The chromatic number of a hypergraph can be bounded using the largest $H$-eigenvalue of the adjacency tensor.

\begin{theorem}[{\citep[][Thm.~3.10]{cooper2012spectra}}]
	For any $m$-uniform hypergraph $\mathcal{H}$, $\chi(\mathcal{H}) \leq \lambda_{\max}(\mathcal{A}_{\mathcal{H}}) +1$.
\end{theorem}

Following our discussion and Lemma~\ref{lemma:H-eigenvalue_Z-tensor}, $\lambda_{\max}(\mathcal{A}_{\mathcal{H}}) = - \lambda_{\min}(-\mathcal{A}_{\mathcal{H}})$. Next we use this approach to compute an upper bound for the chromatic number of a 3-uniform hypergraph.

\begin{example}[Hypergraph application]
	\label{example:hypergraph}
	Let $\mathcal{H}$ be a $3$-uniform hypergraph whose vertex set and edge set are $V(\mathcal{H}) = \{ 1, 2,3,4\}$ and $E(\mathcal{H}) = \{123, 134\}$, respectively. Then $\mathcal{A}_{\mathcal{H}} = (a_{i_1 i_2 i_3})$ is a symmetric tensor in $\mathbb{S}_{3,4}$, where $a_{i_1 i_2 i_3} = 1/2$ if $\{i_1, i_2, i_3\} = \{ 1, 2, 3 \}$ or $\{i_1, i_2, i_3\} = \{ 1, 3 ,4 \}$, and $a_{i_1 i_2 i_3} = 0$ otherwise. With~\eqref{eq:M-tensor_eigenvalue_compute}, one can obtain $\lambda_{\max}(\mathcal{A}_{\mathcal{H}}) = - \lambda_{\min}(-\mathcal{A}_{\mathcal{H}}) = 1.5874$. Thus, the upper bound of $\chi(\mathcal{H})$ is 2.5874. Actually for the given $\mathcal{H}$, $\chi(\mathcal{H}) = 2$. Figure \ref{figure:hypergraph} is the colored $\mathcal{H}$. Node $1$ and $4$ are colored with gray while node $2$ and $3$ are colored with black. For edge $123$ and $134$, node $1$ and node $3$ are with different colors. Thus this is a kind of proper $2$-coloring for $\mathcal{H}$. On the other hand, it is impossible to have a proper $1$-coloring for $\mathcal{H}$. Thus, $\chi(\mathcal{H}) = 2$.
	
	\begin{figure}[bph!]
		\caption{$3$-uniform hypergraph $\mathcal{H}$ in Example \ref{example:hypergraph}}
		\includegraphics[scale=0.6]{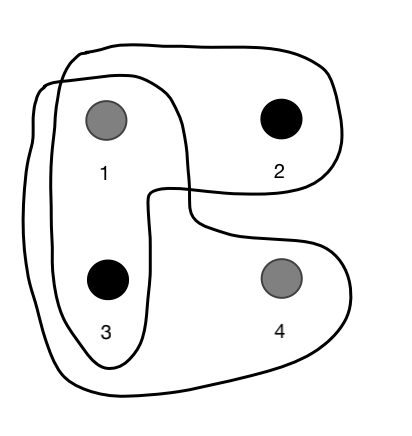}
		\label{figure:hypergraph}
		\centering
	\end{figure}	
	
\end{example}

The analysis of properties of hypergraphs, such as their chromatic number, arises when modeling problems in areas as varied as 
informatics, transportation, molecular biology, and telecommunications, to name just a few~\citep[see, e.g.,][]{bretto2013applications}.

Note that $Z$-tensors with positive minimum $H$-eigenvalue are strong $M$-tensors. Thus, computing the minimum $H$-eigenvalue of $Z$-tensors is also useful in deciding if a $Z$-tensor is a strong $M$-tensor. This question arises when one is interested in finding the sparsest solutions to tensor complementarity problems. Specifically, The authors in \citep{luo2017sparsest} propose the following optimization problem to find one of the sparsest solutions to a tensor complementarity problem	
\begin{equation}
	\label{eq:sparsest1}
	\text{min:} \  \| x \|_0, \ \text{s.t.}\  \mathcal{A} x^{m-1} - b \geq 0 ,\ x \geq 0 , \ x^\intercal (\mathcal{A}x^{m-1} - b) = 0.
\end{equation}	
The objective of~\eqref{eq:sparsest1} is written using the nuclear (i.e., $l_0$) norm. In~\citep{luo2017sparsest}, it is shown that  if $\mathcal{A}$ is a $Z$-tensor, then a sparsest solution of the above tensor complementarity problem can be obtained  by
solving the following polynomial optimization problem

\begin{equation}
	\label{eq:sparsest2}
	\text{min:} \ \|x \|_1, \ \text{s.t.} \ \mathcal{A}x^{m-1} = b, \ x\geq 0.
\end{equation}

Furthermore, they show that if $\mathcal{A}$ is a strong $M$-tensor, problem \eqref{eq:sparsest2} is uniquely solvable and the unique solution is also an optimal solution to problem \eqref{eq:sparsest1}. Besides, when $\mathcal{A}$ is a strong $M$-tensor, the authors in \citep{ding2016solving} propose algorithms which can solve problem \eqref{eq:sparsest2} in polynomial time. Thus, it is helpful to check if $\mathcal{A}$ in problem \eqref{eq:sparsest1} is a strong $M$-tensor so that one can solve it efficiently. In Example \ref{example:tan}, we illustrate this result by considering a tensor that is commonly used in the related literature.

\begin{example}[Sparse solutions of multilinear systems of equations]
	\label{example:tan}
	Let $\mathcal{A}$ in~\eqref{eq:sparsest1} be the $Z$-tensor given by 
		$\mathcal{A} = s \mathcal{I} - \mathcal{D}$,
	where $\mathcal{I}\in \mathbb{S}_{3,2}$ is a diagonal tensor and $\mathcal{D} =(d_{i_1, i_2, i_3}) \in \mathbb{S}_{3,2}$ with $d_{i_1, i_2, i_3} = |\text{tan}(i_1 + i_2 + i_3)|$. Let $s = (1+ \alpha)*\max_{1\leq i\leq 2}(\mathcal{D} {\bf{e}}^{m-1})$, where ${\bf{e}} $ is vector of ones in dimension $2$ and $a = 0.01$. This example is introduced in~\citep[][Ex.~1]{liang2021alternating}. $\mathcal{A}$ is a symmetric tensor, and using~\eqref{eq:M-tensor_eigenvalue_compute}, we find that $\lambda_{\min}(\mathcal{A}) = 1.1538$, which implies that $\mathcal{A}$ is a strong $M$-tensor. Thus, from~\citep[][Thm. 3.2]{ding2016solving}, if $b$ in~\eqref{eq:sparsest1} is positive, one can solve problem~\eqref{eq:sparsest1} in polynomial time by solving problem~$\eqref{eq:sparsest2}$ and the solution is unique.
\end{example}

Furthermore, Corollary \ref{cor:Meigen} can also be used to obtain lower bounds for the minimum $H$-eigenvalues of general $M$-tensors in polynomial time by solving a power cone optimization problem. This follows from the fact that the minimum $H$-eigenvalues of an $M$-tensor is always greater than or equal to the minimum $H$-eigenvalues of its corresponding symmetrized tensor. We prove this fact in the discussion next.
\begin{lemma}[{\citep[][Lem.~2.3]{zhang2014m}}]
	\label{lemma:min_eigenvalue_equation}
	For $m,n \in \mathbb{N}$ and $M$-tensor $\mathcal{A} \in \mathbb{S}_{m,n}$,
	\begin{equation}
		\label{prob:min_H_eigenvalue1}
		\lambda_{\min}(\mathcal{A}) = {\min} \left\{  \mathcal{A} x^{m}: x\in \R^n, \sum_{i=1}^n x_i^{m} = 1  \right\}.
	\end{equation}
\end{lemma}

Let $\tau (\mathcal{A}) = \min \{ \text{Re}(\lambda): \lambda \in \sigma(\mathcal{A}) \}$ where $\sigma(\mathcal{A})$ is the set of all the eigenvalues of $\mathcal{A}$.
\begin{lemma}[{\citep[][Thm.~3.4(a)]{zhang2014m}}]
	\label{lemma:tau_lambda}
	If $m,n \in \mathbb{N}$ and $M$-tensor $\mathcal{A} \in \mathbb{T}_{m,n}$, then $\tau (\mathcal{A}) $ is an $H$-eigenvalue of $\mathcal{A}$. That is~$\lambda_{\min}(\mathcal{A}) = \tau (\mathcal{A})$.
\end{lemma}
In light of Lemma \ref{lemma:tau_lambda}, in what follows, we use $\lambda_{\min}(\mathcal{A})$ to refer to $\tau (\mathcal{A})$ for an $M$-tensor $\mathcal{A}$.

\begin{lemma}
	\label{lemma: sym_minimum_eigenvalue}
	If $\mathcal{A} \in  \mathbb{T}_{m,n}$, then 
	\begin{equation}
		\lambda_{\min}(\mathcal{A}) \geq  \lambda_{\min}(\text{sym}(\mathcal{A} )).
	\end{equation}
\end{lemma}
\begin{proof}
	From the definition of $H$-eigenvalue of a tensor (see Section~\ref{sec:introduction}), if the real value $\lambda$ is an $H$-eigenvalue of $\mathcal{A}$, then there exists $x \in \mathbb{R}^n \backslash \{0 \}$ such that 
	\begin{equation*}
		\mathcal{A}x^{m-1} = \lambda x^{[m-1]}.
	\end{equation*}
	Thus, $\lambda$ satisfies $	\mathcal{A}x^{m} = \lambda  \sum_{i=1}^n x_i^{m}$. When $m$ is even, then clearly $ \sum_{i=1}^n x_i^{m} > 0$. When $m$ is odd, if $\sum_{i=1}^n x_i^{m} < 0$, one can set $y = -x$. Then $y$ and $\lambda $ satisfy
	\[
	\sum_{i=1}^n y_i^{m} > 0, \quad \mathcal{A}y^{m} = \lambda  \sum_{i=1}^n y_i^{m}.
	\]
	Thus, for each $H$-eigenvalue $\lambda$ of $\mathcal{A}$, there exists $x \in \mathbb{R}^n \backslash \{0 \}$ such that
	\[
	\sum_{i=1}^n x_i^{m} > 0, \quad \mathcal{A}x^{m} = \lambda  \sum_{i=1}^n x_i^{m}.
	\]
	Following this result, we have
	
	\begin{equation*}
		\lambda \geq {\min} \left\{  \mathcal{A} x^{m}: x\in \R^n, \sum_{i=1}^n x_i^{m} = 1  \right\}.
	\end{equation*}
	Furthermore, $\lambda \geq 	\lambda_{\min}(\text{sym}(\mathcal{A} ))$ from Lemma~\ref{lemma:min_eigenvalue_equation}.
\end{proof}

To show the performance of the proposed method in obtaining lower bounds for the minimum $H$-eigenvalue of general $M$-tensors, we apply it to compute the lower bounds of the minimum $H$-eigenvalues of the $M$-tensors considered in 
Example~3.1 and Example~3.2 in~\citep{huang2018some} (i.e., different from the tensors in Table~\ref{tab:bounds}, here the tensors are directly taken from~\citep{huang2018some}). Namely, in Table~\ref{tab:bounds_asy}, we list the best upper and lower bounds for the minimum $H$-eigenvalue of the $M$-tensors using the methods proposed in~\citep{huang2018some}, versus the lower bounds of the minimum $H$-eigenvalue of these $M$-tensors obtained using~\eqref{eq:M-tensor_eigenvalue_compute}.


\begin{table}[!htb]
	\begin{center}
		\begin{tabular}{
				>{\centering\arraybackslash}m{1.8cm}
				>{\centering\arraybackslash}m{0.2cm}
				>{\centering\arraybackslash}m{0.2cm}
				>{\centering\arraybackslash}m{1.8cm}
				>{\centering\arraybackslash}m{1.8cm}
				>{\centering\arraybackslash}m{1.8cm}}
			\toprule
			& & & \multicolumn{3}{c}{minimum $H$-eigenvalue}\\
			\cmidrule{4-6}
			\multicolumn{1}{c}{$M$-tensor} &  &  & \multicolumn{1}{c}{best lower bound} & \multicolumn{1}{c}{lower bound} & \multicolumn{1}{c}{best upper bound} \\
			\multicolumn{1}{c}{} & \multicolumn{1}{c}{$m$} & \multicolumn{1}{c}{$n$} &  \multicolumn{1}{c}{\citep{huang2018some}} & \multicolumn{1}{c}{ \eqref{eq:M-tensor_eigenvalue_compute}}& \multicolumn{1}{c}{\citep{huang2018some}}  \\
			\midrule
			Example~3.1 \newline in~\citep{huang2018some} & 3 & 3 & 3.0738 & 4.4404 & 6.8390 \\
			Example~3.2 \newline in~\citep{huang2018some} & 3 & 3 & 4.0768 & 6.3122 & 9.0313 \\
			\bottomrule
		\end{tabular}
	\end{center}
	\caption{Minimum $H$-eigenvalues of asymmetric $M$-tensors. \label{tab:bounds_asy}}
\end{table}

Table \ref{tab:bounds_asy} shows that the lower bounds obtained using~\eqref{eq:M-tensor_eigenvalue_compute} are much tighter than the lower bounds obtained in~\citep{huang2018some}. This empirically indicates that the proposed lower bound is able to provide high quality bounds when comparing with the methods presented in~\citep{huang2018some}.

As another application of the results above, one can verify that an asymmetric $Z$-tensor $\mathcal{A}$ is a strong $M$-tensor by computing the minimum $H$-eigenvalue of $\text{sym}(\mathcal{A} )$. From Lemma \ref{lemma: sym_minimum_eigenvalue}, if $ \lambda_{\min}(\text{sym}(\mathcal{A} )) > 0 $, then $\lambda_{\min}(\mathcal{A}) > 0$. Thus, $\mathcal{A}$ is also a strong $M$-tensor. In Example~\ref{example:pdf}, we use this fact to validate that a tensor that arises when  numerical solving a partial differential equation is a strong $M$-tensor.

\begin{example}[Solutions of multilinear systems of equations]
	\label{example:pdf}
	Consider the problem of numerically solving the Klein–Gordon equation \citep{matsuno1987exact,rheinboldt1998methods}:
	\[ \begin{cases} 
		u(x)^{m-2}\cdot \Delta u(x) = -f(x), &  \quad  \text{in } \Omega, \\
		u(x) = g(x), & \quad \text{on }\partial \Omega,
	\end{cases}
	\]
	where $\Delta = \sum_{k=0}^{d} (\partial^2 / \partial x_k^2)$, $\Omega = [0,1]^d$, and $m = 3, 4, \dots$. When $d=1$, this Klein–Gordon equation can be discretized as the following multilinear system
	\begin{equation*}
		\mathcal{L}_h x^{m-1} = f, 
	\end{equation*}
	in which $h = 1/(n-1)$ and $\mathcal{L}_h = (({\mathcal{L}_h})_{i_1 i_2, \dots, i_m}) \in \mathbb{T}_{m,n}$ with entries defined by
	
%
	
	\[
	\begin{cases} 
		({\mathcal{L}_h})_{1,1, \dots, 1} = ({\mathcal{L}_h})_{n,n, \dots, n} = 1/h^2, & \\
		({\mathcal{L}_h})_{i,i, \dots, i} = 2/h^2, & \text{for all } i = 2, 3, \dots, n-1, \\
		({\mathcal{L}_h})_{i, i-1, i, \dots, i} = -1/h^2(m-1), & \text{for all } i = 2, 3, \dots, n-1, \\
		({\mathcal{L}_h})_{i ,i,i-1 , \dots, i} = -1/h^2(m-1), & \text{for all } i = 2, 3, \dots, n-1, \\
		({\mathcal{L}_h})_{i,i, i, \dots, i-1} = -1/h^2(m-1), & \text{for all } i = 2, 3, \dots, n-1, \\
		({\mathcal{L}_h})_{i, i+1, i, \dots, i} = -1/h^2(m-1), & \text{for all } i = 2, 3, \dots, n-1, \\
		({\mathcal{L}_h})_{i ,i,i+1 , \dots, i} = -1/h^2(m-1), & \text{for all } i = 2, 3, \dots, n-1, \\
		({\mathcal{L}_h})_{i,i, i, \dots, i+1} = -1/h^2(m-1), & \text{for all } i = 2, 3, \dots, n-1.
	\end{cases}
	\]
	
	When $m = 3$ and $n = 3$, with ~\eqref{eq:M-tensor_eigenvalue_compute}, we obtain that $\lambda_{\min}(\text{sym}(\mathcal{L}_{0.5} )) = 2.6667 $. Thus, using Lemma~\ref{lemma: sym_minimum_eigenvalue}, $\lambda_{\min}(\mathcal{L}_{0.5} )  \geq \lambda_{\min}(\text{sym}(\mathcal{L}_{0.5} ))  >0$. As a result, $\mathcal{L}_{0.5}$ is actually a strong $M$-tensor and the algorithms in \citep[][among others]{ding2016solving, rheinboldt1998methods,wang2019neural} can be used to solve the discretized multilinear system.
\end{example}

We subsequently explore further applications within the context of high-order Markov chains~\citep{chung1967markov}.

\begin{example}[Application in high-order Markov chains]
	\label{example:markov_chain}
	The authors in~\citep{li2014limiting} introduce an approximated high-order Markov chain model, described as follows.
	
	\begin{equation}
		\label{eq:mark_chain_main}
		\mathcal{P}x^{m-1} = x, \quad ||x||_{1} = 1,
	\end{equation}
	where $\mathcal{P} = (p_{i_{1}i_{2} \cdots i_{m}}) \in \mathbb{T}_{m,n}$ representing an $(m-1)$th order Markov chain, which is called an $m$-order $n$-dimensional transition probability tensor. That is,
	
	\begin{equation}
		p_{i_{1}i_{2}\cdots i_{m}} \geq 0, \quad \sum_{i_{1} \in [n]} p_{i_{1}i_{2}\cdots i_{m}} = 1,
	\end{equation}
	and $x$ is called a vector of random variables with $x_{i} \geq 0$ and $\sum_{i \in [n]} x_{i} = 1$. It was pointed out in~\citep{liu2018tensor} that the nonlinear system~\eqref{eq:mark_chain_main} is equivalent to the following equation:
	
	\begin{equation}
		\label{eq:mark_chain_main2}
		\begin{cases}
			(\mathcal{I} - \beta \mathcal{P})x^{m-1} = x^{[m-1]} - \beta x, \\
			||x||_{1} = 1.
		\end{cases}
	\end{equation}	
	If we choose the parameter $\beta$ such that $\mathcal{I} - \beta \mathcal{P}$ is a strong $M$-tensor, then we may use the tensor splitting method proposed in~\citep{liu2018tensor} to solve the nonlinear equation.
	We take two examples $\mathcal{P}_1$ and $\mathcal{P}_2$ in~\citep{liu2018tensor,li2014limiting}; also see~\citep{raftery1985model}. The first example $\mathcal{P}_1$ pertains to inter-personal relationships, while the second $\mathcal{P}_2$ is derived from occupational mobility data for physicists. The two transition probability tensors are defined by 
	\begin{itemize}
		\item[]	\begin{align*}  &\mathcal{P}_1(:,:,1) = 
			\begin{pmatrix}
				0.5810 & 0.2432 & 0.1429 \\
				0 & 0.4109 & 0.0701 \\
				0.4190 & 0.3459 & 0.7870 
			\end{pmatrix}, 
			\mathcal{P}_1(:,:,2) = 
			\begin{pmatrix}
				0.4708 & 0.1330 & 0.0327 \\
				0.1341 & 0.5450 & 0.2042 \\
				0.3951 & 0.3220 & 0.7631 
			\end{pmatrix}, \\
			&\mathcal{P}_1(:,:,3) = 
			\begin{pmatrix}
				0.4381 & 0.1003 & 0 \\
				0.0229 & 0.4338 & 0.0930 \\
				0.5390 & 0.4659 & 0.9070 
			\end{pmatrix}. 
		\end{align*}
		\item[]  \begin{align*}	&\mathcal{P}_2(:,:,1) = 
			\begin{pmatrix}
				0.9000 & 0.3340 & 0.3106 \\
				0.0690 & 0.6108 & 0.0754 \\
				0.0310 & 0.0552 & 0.6140 
			\end{pmatrix},
			\mathcal{P}_2(:,:,2) = 
			\begin{pmatrix}
				0.6700 & 0.1040 & 0.0805 \\
				0.2892 & 0.8310 & 0.2956 \\
				0.0408 & 0.0650 & 0.6239 
			\end{pmatrix}, \\
			&\mathcal{P}_2(:,:,3) = 
			\begin{pmatrix}
				0.6604 & 0.0945 & 0.0710 \\
				0.0716 & 0.6133 & 0.0780 \\
				0.2680 & 0.2922 & 0.8501 
			\end{pmatrix}.
		\end{align*}
	\end{itemize}
\end{example}
respectively. Their orders $m$ are both 3 and their numbers of states $n$ are both~3. It was pointed out in~\citep{li2014limiting} that the model~\eqref{eq:mark_chain_main} has a unique positive solution for the two transition probability tensors. Using Corollary~\ref{cor:Meigen}, we find that when $\beta = 0.3$, $\lambda_{\min}(\text{sym}(I - \beta \mathcal{P}_1)) = 0.081666 $ and $\lambda_{\min}(\text{sym}(I - \beta \mathcal{P}_2)) =0.099781 $ . Thus, using Lemma~\ref{lemma: sym_minimum_eigenvalue}, $\lambda_{\min}(I - \beta \mathcal{P})  \geq \lambda_{\min}(\text{sym}(I - \beta \mathcal{P} ))  >0$ for $\mathcal{P} = \mathcal{P}_1$ and $\mathcal{P} = \mathcal{P}_2$. The tensor splitting method proposed in~\citep{liu2018tensor} can be utilized to solve the nonlinear equation~\eqref{eq:mark_chain_main2} under these conditions.



\section{Minimum $H$-eigenvalue of the Fan product of symmetric $M$-tensors}
\label{section:fan_prod}
For two tensors $\mathcal{A} = (a_{i_1 i_2 \dots i_m}) \in \mathbb{T}_{m,n}$ and $\mathcal{B} = (b_{i_1 i_2 \dots i_m}) \in \mathbb{T}_{m,n}$, their Fan product $\mathcal{A} \star \mathcal{B}$ is a tensor defined by 
\begin{equation}
	(\mathcal{A} \star \mathcal{B})_{i_1i_2\dots i_m} = (-1)^{\delta_{i_1i_2 \dots i_m} + 1} a_{i_1 i_2 \dots i_m}b_{i_1 i_2 \dots i_m},
\end{equation}
where $\delta_{i_1i_2 \dots i_m}=1$ if $i_1 = i_2 = \cdots = i_m$, and $\delta_{i_1i_2 \dots i_m} = 0$ otherwise. 

One of the main characteristics of the Fan product is that the Fan product of $M$-tensors is also an $M$-tensor \citep{shen2019some}. The authors in \citep{cheng2014new,fang2007bounds,shen2019some} propose bounds for the minimum $H$-eigenvalue of $\mathcal{A} \star \mathcal{B}$ where $\mathcal{A}$ and $\mathcal{B}$ are $Z$-matrices ($Z$-tensors). With the help of the proposed new characterisation of symmetric $M$-tensors, we provide tighter bounds for the minimum $H$-eigenvalue of the Fan product of symmetric $M$-tensors. Specifically, we show both theoretically and empirically that our proposed lower bounds are tighter than any of the bounds provided in \citep{shen2019some}.

For a symmetric $M$-tensor $\mathcal{A}$, it follows from Theorem \ref{theo:gdd_decomp} and \ref{theo:gdd_characterization} and Corollary~\ref{cor:Meigen} that one can write $\mathcal{A}$ as 
\begin{equation}
	\label{eq:tensor_decomp_A}
	\mathcal{A} =  \lambda_{\min}(\mathcal{A})  \mathcal{I} +  \sum_{ \vec{i} \in \mathscr{D}_n^m} \mathcal{A}^{\vec{i}},
\end{equation}
where 
$\mathcal{A}^{\vec{i}} = (a^{\vec{i}}_{j_1 j_2 \dots j_m}) \in GDD^+_{m,n} \cap \mathbb{D}^{\ii}_{m,n}$, for all $\vec{i} \in \mathscr{D}_n^m$. 
Similarly, for another symmetric $M$-tensor $\mathcal{B}$, we can also decompose it as 
\begin{equation}
	\label{eq:tensor_decomp_B}
	\mathcal{B} =  \lambda_{\min}(\mathcal{B})  \mathcal{I} +  \sum_{\vec{i} \in \mathscr{D}_n^m} \mathcal{B}^{\vec{i}},
\end{equation}
where 
$\mathcal{B}^{\vec{i}} = (b^{\vec{i}}_{j_1 j_2 \dots j_m}) \in GDD^+_{m,n} \cap \mathbb{D}^{\ii}_{m,n}$, for all $\vec{i} \in \mathscr{D}_n^m$. From Theorem \ref{th:check_gdd}, the decomposition of $\mathcal{A}$ and $\mathcal{B}$ can be done in polynomial time.

Using this decomposition, we obtain tighter lower bounds of the minimum $H$-eigenvalue for the Fan product of two symmetric $M$-tensors. Before presenting the lower bounds, we need the following result regarding the minimum $H$-eigenvalue of $M$-tensors. 
\begin{lemma}[{\citep[][Lem.~2.2]{shen2019some}}]
	If $\mathcal{A} = (a_{i_1\dots i_m}) \in \mathbb{T}_{m,n}$ is an $M$-tensor, then
	\begin{equation*}
		\min_{1\leq i\leq n} \frac{(\mathcal{A} x^{m-1})_i}{x_i^{m-1}} \leq \lambda_{\min}(\mathcal{A})
	\end{equation*}
	for any $x = (x_1, \dots, x_n)^\intercal \in \mathbb{R}^n_{++}$.
\end{lemma}

For $M$-tensor $\mathcal{A} = (a_{i_1\dots i_m}) \in \mathbb{S}_{m,n}$ and $\mathcal{B} = (b_{i_1\dots i_m}) \in \mathbb{S}_{m,n}$ and decomposition \eqref{eq:tensor_decomp_A}  and \eqref{eq:tensor_decomp_B}, let
\begin{flalign}
	\Omega_1 (\mathcal{A}, \mathcal{B})= \min_{1 \leq i \leq n} \left(a_{ii\dots i} b_{ii\dots i} -  \sum_{(i,i_2,\dots, i_m)  \in \mathscr{D}_n^m}  \frac{ 1}{\binom{m-1}{\alpha^{\vec{i}} -e_i}} {a}_{ii \dots     i}^{i,i_2,\dots, i_m } {b}_{ii \dots     i}^{i,i_2,\dots, i_m } \right) \label{eq:Fan_prod_lb1},
\end{flalign}
\begin{flalign}
	\Omega_2 (\mathcal{A}, \mathcal{B})= \min_{1 \leq i \leq n} \left(a_{ii\dots i} b_{ii\dots i} -  \sum_{(i,i_2,\dots, i_m)  \in \mathscr{D}_n^m}   |{a}_{i i_2\dots i_m} | {b}_{ii \dots     i}^{i,i_2,\dots, i_m } \right) \label{eq:Fan_prod_lb2},
\end{flalign}
\begin{flalign}
	\Omega_2 (\mathcal{B}, \mathcal{A})= \min_{1 \leq i \leq n} \left(a_{ii\dots i} b_{ii\dots i} -  \sum_{(i,i_2,\dots, i_m)  \in \mathscr{D}_n^m}   |{b}_{i i_2\dots i_m} | {a}_{ii \dots     i}^{i,i_2,\dots, i_m } \right) \label{eq:Fan_prod_lb21},
\end{flalign}
and 

\begin{flalign}
	& \Omega_3 (\mathcal{A}, \mathcal{B})= \min_{1 \leq i \leq n} \bigg(a_{ii\dots i} b_{ii\dots i} -  \nonumber \\
	& \qquad \sum_{(i,i_2,\dots, i_m)\in \mathscr{D}_n^m}   \big(|{a}_{i i_2\dots i_m}|)^{\frac{1}{2}} (|{b}_{i i_2\dots i_m}|)^{\frac{1}{2}} ({a}_{ii \dots     i}^{i,i_2,\dots, i_m })^{\frac{1}{2}}  ({b}_{ii \dots     i}^{i,i_2,\dots, i_m })^{\frac{1}{2}} \bigg) \label{eq:Fan_prod_lb3}.
\end{flalign}

Let $\alpha_i(\mathcal{A}) = \max_{(i_2,\dots , i_m) \neq (i, \dots, i)} |a_{i i_2 \dots i_m}|$ for $i \in [n]$,
\begin{equation}
	\Gamma_1(\mathcal{A}, \mathcal{B}) = \min_{1 \leq i \leq n}\{a_{ii \dots i} \lambda_{\min}(\mathcal{B}) + b_{ii \dots i}\lambda_{\min}(\mathcal{A})  \} - \lambda_{\min}(\mathcal{A}) \lambda_{\min}(\mathcal{B}),
	\label{eq:Fan_prod_lb1_comp}
\end{equation}
\begin{equation}
	\Gamma_2(\mathcal{A}, \mathcal{B}) = \min_{1\leq i \leq n} \{ a_{ii\dots i} b_{ii\dots i} - \alpha_i(\mathcal{A})(b_{ii \dots i} -\lambda_{\min}(\mathcal{B}))  \},
	\label{eq:Fan_prod_lb2_comp}
\end{equation}
\begin{equation}
	\Gamma_2(\mathcal{B}, \mathcal{A}) = \min_{1\leq i \leq n} \{ a_{ii\dots i} b_{ii\dots i} - \alpha_i(\mathcal{B})(a_{ii \dots i} -\lambda_{\min}(\mathcal{A}))  \},
	\label{eq:Fan_prod_lb21_comp}
\end{equation}
and
\begin{equation}
	\begin{aligned}
		&\Gamma_3(\mathcal{A}, \mathcal{B}) \\
		=& \min_{1\leq i \leq n} \Bigg\{  a_{ii\dots i} b_{ii\dots i} \\
		& - (\alpha_i(\mathcal{A}))^{\frac{1}{2}} (\alpha_i(\mathcal{B}))^{\frac{1}{2}} (a_{ii \dots i} - \lambda_{\min}(\mathcal{A}))^{\frac{1}{2}} (b_{ii \dots i} - \lambda_{\min}(\mathcal{B}))^{\frac{1}{2}}  \Bigg\}.
	\end{aligned}
	\label{eq:Fan_prod_lb3_comp}
\end{equation}
Expressions \eqref{eq:Fan_prod_lb1_comp}, \eqref{eq:Fan_prod_lb2_comp}, \eqref{eq:Fan_prod_lb21_comp} and \eqref{eq:Fan_prod_lb3_comp} are proposed in \citep{shen2019some} as lower bounds for the minimum $H$-eigenvalue of the Fan product of $M$-tensors. In Theorem \ref{them:Fan_prod_lb}, we prove that the expressions~\eqref{eq:Fan_prod_lb1},~\eqref{eq:Fan_prod_lb2},~\eqref{eq:Fan_prod_lb21} and~\eqref{eq:Fan_prod_lb3} tighten these lower bounds. As it will be illustrated in Table \ref{tab:fan_prod}, which of the expression~\eqref{eq:Fan_prod_lb1},~\eqref{eq:Fan_prod_lb2},~\eqref{eq:Fan_prod_lb21} and~\eqref{eq:Fan_prod_lb3} provides the best lower bound depends on the specific tensors being considered.
\begin{theorem}
	\label{them:Fan_prod_lb}
	For symmetric $M$-tensors $\mathcal{A}$ and $\mathcal{B}$,
	\begin{enumerate}[label = (\roman*)]
		\item \label{label:lb1} $\Gamma_1(\mathcal{A}, \mathcal{B}) \leq  \Omega_1 (\mathcal{A}, \mathcal{B})\leq  \lambda_{\min}(\mathcal{A} \star \mathcal{B})  $ 
		\item \label{label:lb2}  $\Gamma_2(\mathcal{A}, \mathcal{B}) \leq  \Omega_2 (\mathcal{A}, \mathcal{B}) \leq  \lambda_{\min}(\mathcal{A} \star \mathcal{B}) $,   $\Gamma_2(\mathcal{B}, \mathcal{A}) \leq  \Omega_2 (\mathcal{B}, \mathcal{A}) \leq  \lambda_{\min}(\mathcal{A} \star \mathcal{B}) $
		\item   \label{label:lb3}  $\Gamma_3(\mathcal{A}, \mathcal{B}) \leq  \Omega_3 (\mathcal{A}, \mathcal{B}) \leq  \lambda_{\min}(\mathcal{A} \star \mathcal{B}) $ 
	\end{enumerate}
\end{theorem}
\begin{proof}
	First, for symmetric $M$-tensors $\mathcal{A}$ and $\mathcal{B}$, one can derive the decompositions given in \eqref{eq:tensor_decomp_A} and \eqref{eq:tensor_decomp_B}. For each $\vec{i} = (i_1, i_2, \dots ,i_m)  \in \mathscr{D}_n^m$ with $\alpha^{\vec{i}}$ as the tight power, since $\mathcal{A}^{\vec{i}} = (a^{\vec{i}}_{j_1 j_2 \dots j_m}) \in GDD^+_{m,n} \cap \mathbb{D}^{\ii}_{m,n}$, there exist $u_i > 0$ for $i \in [n]$ such that
	\begin{flalign}
		{a}_{ii \cdots     i}^{\vec{i}} u_i^m \geq \binom{m-1}{\alpha^{\vec{i}} -e_i} |  {a}_{\vec{i}}^{\vec{i}} u_{i_1} \cdots u_{i_m}|.
	\end{flalign}
	Similarly, there exist $v_i > 0$ for $i \in [n]$ such that
	\begin{flalign}
		{b}_{ii \cdots     i}^{\vec{i}} v_i^m \geq \binom{m-1}{\alpha^{\vec{i}} -e_i} |  {b}_{\vec{i}}^{\vec{i}} v_{i_1} \cdots v_{i_m}|.
	\end{flalign}
	To show \ref{label:lb1}, let $z = (z_1, z_2,\dots, z_n)^\intercal \in \mathbb{R}^n_{++}$, where $z_i = u_iv_i$ for all $i \in [n]$. Then for $i \in [n]$,
	\begin{flalign*}
		&\frac{((\mathcal{A}\star \mathcal{B} )z^{m-1})_i}{z_i^{m-1}} \\
		=& a_{ii \cdots i} b_{ii \cdots i} -  \frac{1}{z_i^{m-1}} \left( \sum_{( i_2, \dots, i_m) \neq (i,\dots,i)} a_{ii_2 \dots i_m}b_{i i_2\dots i_m}z_{i_2} \cdots z_{i_m} \right) \\
		=& a_{ii \cdots i} b_{ii \cdots i} -   \frac{1}{(u_i v_i)^{m-1}}\sum_{(i_2, \dots, i_m) \neq (i,i,\dots,i)} a_{ii_2 \dots i_m}^{i, i_2, \dots, i_m}u_{i_2} \cdots u_{i_m}b_{i i_2 \dots i_m}^{i, i_2, \dots, i_m}v_{i_2} \cdots v_{i_m}\\
		\geq &a_{ii \cdots i} b_{ii \cdots i} - \left(  \sum_{(i,i_2,\dots, i_m)  \in \mathscr{D}_n^m}  \frac{ 1}{\binom{m-1}{\alpha^{\vec{i}} -e_i}} {a}_{ii \cdots     i}^{i,i_2,\dots, i_m } {b}_{ii \cdots     i}^{i,i_2,\dots, i_m }  \right) \\
		\geq & a_{ii \cdots i} b_{ii \cdots i} - \left(  \sum_{(i,i_2,\dots, i_m)  \in \mathscr{D}_n^m}  {a}_{ii \cdots     i}^{i,i_2,\dots, i_m } {b}_{ii \cdots     i}^{i,i_2,\dots, i_m }  \right)  \\
		\geq & a_{ii \cdots i} b_{ii \cdots i} -\left(  \sum_{(i,i_2,\dots, i_m)  \in \mathscr{D}_n^m}  {a}_{ii \cdots     i}^{i,i_2,\dots, i_m }  \right)  \left( \sum_{(i,i_2,\dots, i_m)  \in \mathscr{D}_n^m} {b}_{ii \cdots     i}^{i,i_2,\dots, i_m }   \right) \\
		= &  a_{ii \cdots i} b_{ii \cdots i} - (a_{ii \cdots i } -  \lambda_{\min}(\mathcal{A})) (b_{ii \cdots i }- \lambda_{\min}(\mathcal{B})) \\
		=  & a_{ii \dots i} \lambda_{\min}(\mathcal{B}) + b_{ii \dots i}\lambda_{\min}(\mathcal{A}) - \lambda_{\min}(\mathcal{A}) \lambda_{\min}(\mathcal{B}).
	\end{flalign*}
	The second to last equality follows from the decompositions \eqref{eq:tensor_decomp_A} and \eqref{eq:tensor_decomp_B}.
	Thus, it follows that 
	\begin{equation*}
		\Gamma_1(\mathcal{A}, \mathcal{B})  \leq \Omega_1 (\mathcal{A}, \mathcal{B})\leq  \min_{1\leq i\leq n} \frac{((\mathcal{A}\star \mathcal{B} )z^{m-1})_i}{z_i^{m-1}} \leq  \lambda_{\min}(\mathcal{A} \star \mathcal{B}) .
	\end{equation*}
	To show \ref{label:lb2}, let $z = (z_1, z_2,\dots, z_n)^\intercal \in \mathbb{R}^n_{++}$, where $z_i = u_i$ for all $i \in [n]$. Then for $i \in [n]$,
	\begin{flalign*}
		&\frac{((\mathcal{A}\star \mathcal{B} )z^{m-1})_i}{z_i^{m-1}} \\
		=& a_{ii \cdots i} b_{ii \cdots i} -  \frac{1}{u_i^{m-1}} \left( \sum_{( i_2, \dots, i_m) \neq (i,\dots,i)} a_{i i_2 \dots i_m}b_{i i_2 \dots i_m}u_{i_2} \cdots u_{i_m} \right) \\
		\geq & a_{ii \cdots i} b_{ii \cdots i} - \left( \sum_{(i,i_2,\dots, i_m)  \in \mathscr{D}_n^m}  {a}_{ii \cdots     i}^{i,i_2,\dots, i_m }  |b_{i i_2 \dots i_m}|  \right) \\
		\geq &a_{ii \cdots i} b_{ii \cdots i} - \alpha_i(\mathcal{B})  \left( \sum_{(i,i_2,\dots, i_m)  \in \mathscr{D}_n^m}  {a}_{ii \cdots     i}^{i,i_2,\dots, i_m }    \right) \\
		=  & a_{ii \cdots i} b_{ii \cdots i} - \alpha_i(\mathcal{B})  \left(  a_{ii\cdots i} - \lambda_{\min}(\mathcal{A})   \right). 
	\end{flalign*}
	Thus,
	\begin{equation}
		\label{eq:Fan_prod_lb_2_1}
		\Gamma_2(\mathcal{B}, \mathcal{A}) \leq  \Omega_2 (\mathcal{B}, \mathcal{A}) \leq  \min_{1\leq i\leq n} \frac{((\mathcal{A}\star \mathcal{B} )z^{m-1})_i}{z_i^{m-1}} \leq 	  \lambda_{\min}(\mathcal{A} \star \mathcal{B}) .
	\end{equation}
	Similarly, it then follows that
	\begin{equation*}
		\Gamma_2(\mathcal{A}, \mathcal{B}) \leq  \Omega_2 (\mathcal{A}, \mathcal{B}) \leq  \min_{1\leq i\leq n} \frac{((\mathcal{A}\star \mathcal{B} )z^{m-1})_i}{z_i^{m-1}} \leq 	  \lambda_{\min}(\mathcal{A} \star \mathcal{B})
	\end{equation*}
	by setting $z_i = v_i$ for all $i \in [n]$ in the proof of \eqref{eq:Fan_prod_lb_2_1}.
	
	To show \ref{label:lb3}, let $z = (z_1, z_2,\dots, z_n)^\intercal \in \mathbb{R}^n_{++}$, where $z_i = (u_i)^{\frac{1}{2}} (v_i)^{\frac{1}{2}}$ for all $i \in [n]$. Then for $i  \in [n]$,
	\begin{flalign*}
		& \frac{((\mathcal{A}\star \mathcal{B} )z^{m-1})_i}{z_i^{m-1}}  \\
		=& a_{ii \cdots i} b_{ii \cdots i} -  \frac{1}{(u_i v_i)^{\frac{m-1}{2}}}  \sum_{( i_2, \dots, i_m) \neq (i,\dots,i)} a_{i i_2 \dots i_m}u_{i_2}^{\frac{1}{2}} \cdots u_{i_m}^{\frac{1}{2}}b_{i i_2 \dots i_m}v_{i_2}^{\frac{1}{2}} \cdots v_{i_m}^{\frac{1}{2}}\\
		\geq & a_{ii \cdots i} b_{ii \cdots i} -    \sum_{(i,i_2,\dots, i_m)  \in \mathscr{D}_n^m} |a_{i i_2 \dots i_m}|^\frac{1}{2} (a_{ii \cdots i}^{i,i_2,\dots, i_m })^\frac{1}{2}  |b_{i i_2 \dots i_m}|^\frac{1}{2} (b_{ii \cdots i}^{i,i_2,\dots, i_m })^\frac{1}{2}  \\ 
		&\text{in what follows,  notice that we use $\newi := (i, i_2, \cdots, i_m)$ for ease of presentation} \\
		\geq & a_{ii \cdots i} b_{ii \cdots i} -   \left( \sum_{\newi := (i,i_2,\dots, i_m)  \in \mathscr{D}_n^m} |a_{\newi }| a_{ii \cdots i}^{\newi  }  \right)^{\frac{1}{2}}   \left( \sum_{\newi :=(i,i_2,\dots, i_m)  \in \mathscr{D}_n^m}|b_{\newi }|b_{ii \cdots i}^{\newi } \right)^{\frac{1}{2}}\\
		\geq & a_{ii \cdots i} b_{ii \cdots i} -    \left( \sum_{\newi :=(i,i_2,\dots, i_m)  \in \mathscr{D}_n^m}   \alpha_i(\mathcal{A}) a_{ii \cdots i}^{\newi  }  \right)^{\frac{1}{2}}   \left( \sum_{\newi :=(i,i_2,\dots, i_m)  \in \mathscr{D}_n^m} \alpha_i(\mathcal{B})b_{ii \cdots i}^{\newi  } \right)^{\frac{1}{2}}\\
		= & a_{ii \cdots i} b_{ii \cdots i} -  (\alpha_i(\mathcal{A}))^{\frac{1}{2}}(a_{ii \cdots i} -  \lambda_{\min}(\mathcal{A}) )^{\frac{1}{2}}  (\alpha_i(\mathcal{B}))^{\frac{1}{2}}(b_{ii \cdots i} -  \lambda_{\min}(\mathcal{B}) )^{\frac{1}{2}} \\
		= & \Gamma_3(\mathcal{A}, \mathcal{B}) . 
	\end{flalign*}
	The second inequality follows from the Cauchy–Schwarz inequality. Thus,
	\begin{equation*}
		\Gamma_3(\mathcal{A}, \mathcal{B}) \leq  \Omega_3 (\mathcal{B}, \mathcal{A}) \leq  \min_{1\leq i\leq n} \frac{((\mathcal{A}\star \mathcal{B} )z^{m-1})_i}{z_i^{m-1}} \leq 	  \lambda_{\min}(\mathcal{A} \star \mathcal{B}) .
	\end{equation*}
	
\end{proof}

To illustrate how the new bounds introduced in Theorem \ref{them:Fan_prod_lb} tighten the bounds introduced in \citep{shen2019some}, we compute bounds proposed here (i.e., \eqref{eq:Fan_prod_lb1},~\eqref{eq:Fan_prod_lb2}~\eqref{eq:Fan_prod_lb21} and~\eqref{eq:Fan_prod_lb3}) on the minimum $H$-eigenvalue of the Fan product of the symmetrized tensors in Example~3.9 in~\citep{shen2019some} and compare the results with the bounds (including \eqref{eq:Fan_prod_lb1_comp}, \eqref{eq:Fan_prod_lb2_comp}, \eqref{eq:Fan_prod_lb21_comp} and \eqref{eq:Fan_prod_lb3_comp}) proposed in~\citep{shen2019some}.

\begin{example} In this example, a tensor $\mathcal{A} = (a_{i_1 i_2 i_3 i_4}) \in \mathbb{R}^{[4,2]}$ is written in unfolded form as
	\begin{flalign*}
		\mathcal{A} = 	\left[
		\begin{array}{cc|cc|cc|cc}
			a_{1111} & a_{1211} & a_{1112} & a_{1212} & a_{1121}  & a_{1221} & a_{1122} & a_{1222} \\
			a_{2111} & a_{2211} & a_{2112} & a_{2212} & a_{2121}  & a_{2221} & a_{2122}& a_{2222}
		\end{array}
		\right].
	\end{flalign*}
	
	Symmetric $\mathcal{M}$-tensors $\mathcal{A}_i$, $\mathcal{B}_i \in \mathbb{R}^{[4,2]}$ for $i = 1,2,3$ are given as follows 
	\begin{flalign*}
		\mathcal{A}_1 = 	\left[
		\begin{array}{cc|cc|cc|cc}
			3 & -0.5 & -0.5 & 0 & -0.5  & 0 & 0 & -0.25 \\
			-0.5 & 0 & 0 & -0.25 & 0  & -0.25 & -0.25& 2
		\end{array}
		\right],
	\end{flalign*}
	\begin{flalign*}
		\mathcal{B}_1 = 	\left[
		\begin{array}{cc|cc|cc|cc}
			1.5 & -0.125 & -0.125 & 0 & -0.125  & 0 & 0 & -0.625 \\
			-0.125 & 0 & 0 & -0.625 & 0  & -0.625 & -0.625& 2.5
		\end{array}
		\right],
	\end{flalign*}
	\begin{flalign*}
		\mathcal{A}_2 = 	\left[
		\begin{array}{cc|cc|cc|cc}
			3.8 & -0.5 & -0.5 & -13/30 & -0.5  & -13/30 & -13/30 & -0.5 \\
			-0.5 & -13/30 & -13/30 & -0.5 & -13/30  & -0.5 & -0.5& 3.9
		\end{array}
		\right],
	\end{flalign*}
	\begin{flalign*}
		\mathcal{B}_2 = 	\left[
		\begin{array}{cc|cc|cc|cc}
			3.2 & -0.675 & -0.675 & -1/3 & -0.675  & -1/3 & -1/3 & -0.35 \\
			-0.675 & -1/3 & -1/3 & -0.35 & -1/3  & -0.35 & -0.35& 3.9
		\end{array}
		\right],
	\end{flalign*}
	\begin{flalign*}
		\mathcal{A}_3 = 	\left[
		\begin{array}{cc|cc|cc|cc}
			3.8 & -0.575 & -0.575 & -11/30 & -0.575  & -11/30 & -11/30 & -0.4 \\
			-0.575 & -11/30 & -11/30 & -0.4 & -11/30  & -0.4 & -0.4& 3.7
		\end{array}
		\right],
	\end{flalign*}
	\begin{flalign*}
		\mathcal{B}_3 = 	\left[
		\begin{array}{cc|cc|cc|cc}
			3.5 & -0.35 & -0.35 & -23/60 & -0.35  & -23/60 & -23/60 & -0.525 \\
			-0.35 & -23/60 & -23/60 & -0.525 & -23/60  & -0.525 & -0.525& 3.1
		\end{array}
		\right],
	\end{flalign*}
	
	Similar to Table~1 in~\citep{shen2019some}, Table~\ref{tab:fan_prod} shows the bounds for the minimum $H$-eigenvalue of the Fan product of $\mathcal{A}_i$ and~$\mathcal{B}_i$, $i = 1,2,3$ obtained with the expressions \eqref{eq:Fan_prod_lb1}, \eqref{eq:Fan_prod_lb2}, \eqref{eq:Fan_prod_lb21}, \eqref{eq:Fan_prod_lb3} and the lower bound expressions from~\citep{shen2019some}. For $\mathcal{A}_i$ and~$\mathcal{B}_i$, $i=1,2,3$, the lower bounds from~\eqref{eq:Fan_prod_lb1} are 4.0717, 10.8346 and 10.3187 respectively which are larger than the values from all the bounds proposed in~\citep{shen2019some}. Besides, the lower bounds from~\eqref{eq:Fan_prod_lb2},~\eqref{eq:Fan_prod_lb21} and~\eqref{eq:Fan_prod_lb3} are also tighter than the lower bounds from (3.4), (3.5) and (3.6)  (i.e., \eqref{eq:Fan_prod_lb2_comp}, \eqref{eq:Fan_prod_lb21_comp}, and\eqref{eq:Fan_prod_lb3_comp}) in~\citep{shen2019some}, respectively. This empirically validates the fact that the proposed bounds~\eqref{eq:Fan_prod_lb1}, \eqref{eq:Fan_prod_lb2}, \eqref{eq:Fan_prod_lb21}, \eqref{eq:Fan_prod_lb3} are tighter lower bounds for the minimum $H$-eigenvalue of the Fan product of two symmetric $M$-tensors. The proposed bounds \eqref{eq:Fan_prod_lb1}, \eqref{eq:Fan_prod_lb2}, \eqref{eq:Fan_prod_lb21} and \eqref{eq:Fan_prod_lb3} contain more information comparing the bounds proposed in \citep{shen2019some}. As a result, they are able to provide the tighter lower bounds. Note also that the expression among \eqref{eq:Fan_prod_lb1}, \eqref{eq:Fan_prod_lb2}, \eqref{eq:Fan_prod_lb21} and \eqref{eq:Fan_prod_lb3} that provides the best lower bound depends on the specific tensors. For example in the first column ($\mathcal{A}_1$ and $\mathcal{B}_1$), the best lower bound is given by expression~\eqref{eq:Fan_prod_lb2} while for the second column ($\mathcal{A}_2$ and $\mathcal{B}_2$) it is expression~\eqref{eq:Fan_prod_lb1}.

	\begin{table}[!htb]
		\begin{center}
			\begin{tabular}{lrrr}
				\toprule
				& $\mathcal{A} = \mathcal{A}_1, \mathcal{B} = \mathcal{B}_1$ &$\mathcal{A} = \mathcal{A}_2, \mathcal{B} = \mathcal{B}_2$ &$\mathcal{A} = \mathcal{A}_3, \mathcal{B} = \mathcal{B}_3$   \\
				\midrule
				$ \lambda_{\min}(\mathcal{A}) $  & 0.9723& 0.54995& 0.6970 \\
				$ \lambda_{\min}(\mathcal{B}) $  & 0.5000&0.41253 & 0.3717 \\
				$ \lambda_{\min}(\mathcal{A} \star \mathcal{B}) $  & 4.2762& 11.3818&12.0646  \\
				\midrule
				\multirow{2}{10 em}{Lower bounds on $ \lambda_{\min}(\mathcal{A} \star \mathcal{B})$ from  \citep{shen2019some} }  
				& &  & \\
				& &  & \\
				& &  & \\
				\midrule
				(3.1) in \citep{shen2019some} & 2.4722 & 3.1006   &  3.2768   \\
				(3.3) in \citep{shen2019some} \eqref{eq:Fan_prod_lb1_comp}  & 4.0000 &10.7663 &  9.9012 \\
				(3.4) in \citep{shen2019some} \eqref{eq:Fan_prod_lb2_comp} & 3.2327&9.9662 &  9.8934 \\
				(3.5) in \citep{shen2019some} \eqref{eq:Fan_prod_lb21_comp} & 4.0000&10.7663 & 9.9012   \\
				(3.6) in \citep{shen2019some} \eqref{eq:Fan_prod_lb3_comp} & 3.7040 & 10.4114 &  9.8973 \\
				(3.7) in \citep{shen2019some}  & 2.5000&10.2250 & 2.6294  \\
				\midrule
				\multirow{2}{10 em}{Proposed lower bounds on $ \lambda_{\min}(\mathcal{A} \star \mathcal{B})$ }  
				& &  & \\
				& &  & \\
				& &  & \\
				\midrule
				\eqref{eq:Fan_prod_lb1}  & 4.0717  & 10.8346 &  10.3187   \\
				\eqref{eq:Fan_prod_lb2}  & 4.1562  & 10.8177  &  10.3682  \\
				\eqref{eq:Fan_prod_lb21}  &4.0717 & 10.5605  &  10.1657  \\
				\eqref{eq:Fan_prod_lb3}  & 4.1169  & 10.6959  &  10.2691  \\
				\bottomrule
			\end{tabular}
	\end{center}
	\caption{Lower bounds for the minimum $H$-eigenvalues of the Fan product of symmetric $M$-tensors. \label{tab:bounds_asy_fan}}
	\label{tab:fan_prod}
\end{table}

\end{example}

\section{Conclusions}
\label{part:conclusion}

In this work, a new characterization of symmetric $H^+$-tensors is presented (see Corollary \eqref{col:fullGDDplus}). As a result of this characterization, it follows that one can identify whether a tensor is a symmetric $H^+$-tensor in polynomial time (see Theorem~\ref{th:check_gdd}). Comparing other characterizations which typically focus on sufficient conditions for a tensor to be an $H^+$-tensor, our characterization provides sufficient and necessary conditions. Besides, the set of symmetric $H^+$-tensors is described using tractable convex cones; in particular, the power cone.

We apply the new characterization of symmetric $H^+$-tensors in computing the minimum $H$-eigenvalue of symmetric $M$-tensors. 
In
particular, we compare the best bounds for the minimum $H$-eigenvalues proposed
in the related literature with these $H$-eigenvalues of symmetric $M$-tensors; which
can be computed in polynomial time by solving a power cone optimization problem.
We also show that this approach to computing $H$-eigenvalues of symmetric $M$-tensors is more efficient than using 
homotopy continuation type algorithms that allow the more general computation of complex generalized tensor eigenpairs~\citep{chen2016computing}.
Furthermore, we illustrate how this new characterization of symmetric $H^+$-tensors can be used to obtain tighter lower bounds for the minimum $H$-eigenvalue of the Fan product of two symmetric $M$-tensors. We show both theoretically and empirically that the proposed bounds are tighter compared to the bounds proposed in \citep{shen2019some}. We illustrate the relevance of our results with practical examples drawn from polynomial optimization, hypergraphs analysis, complementarity problems, multilinear systems, and high-order Markov chains. Besides these applications, we believe more interesting results can be obtained with the proposed new characterization of symmetric $H^+$-tensors and $M$-tensors.


%
%



%
\bibliography{references}
\bibliographystyle{plainnat}

\end{document}